\documentclass[11pt,reqno]{amsart}

  \usepackage[margin=1in]{geometry}
 \usepackage{amsmath,amssymb}
 \usepackage{algpseudocode}
 \usepackage{algorithm}
 \usepackage{algorithmicx}
 \usepackage{caption}
 \usepackage{subcaption}
 \usepackage{amsthm}
 \usepackage{esint}
 \numberwithin{equation}{section}
 \usepackage{amsfonts}
 \usepackage{graphicx}
 \usepackage{hyperref}
\usepackage{enumitem}
 \usepackage{makecell}
 \usepackage{longtable}

 \usepackage{forest}
 \usepackage{comment}

\usepackage{amsmath}
\usepackage{amssymb}
\usepackage{amsthm}
\usepackage{amsmath}
\usepackage{setspace}
\usepackage{xcolor}
\usepackage{fancyhdr}
\usepackage{amssymb}
\usepackage{amsthm}
\usepackage{listings}
    \usepackage{cite}
    \usepackage{tabularx}
    \usepackage{graphicx}
    \usepackage{epstopdf}    
    \usepackage{epsfig}
    \usepackage{float}
    \usepackage{ listings}
    \usepackage{appendix}
 \theoremstyle{plain}

 \newtheorem{theorem}{Theorem}[section]
 
 \newtheorem{lemma}[theorem]{Lemma}
 
 \theoremstyle{definition}

 \newtheorem{remark}[theorem]{Remark}

 \let\pa=\partial
 \let\al=\alpha
 \let\b=\beta
 
 \let\g=\gamma
 \let\e=\varepsilon

 \let \kp = \kappa
 \let\lam=\lambda

 \let\f=\frac
 \let \les = \lesssim
  
 \let\om=\omega
 
 \let \th = \theta

 \let \vp = \varphi
 
\let\B = \Big
 \let\D=\Delta

 \let\td = \widetilde
 
 \let\wh=\widehat
 
 \let \olin = \overline

 \let\teq \triangleq
 
 \let\pa=\partial
 
 \let \vs = \vspace

 \def\cE{{\mathcal E}}
 \def\cF{{\mathcal F}}

 \def\cK{{\mathcal K}}
 \def\cL{{\mathcal L}}
 
 \def\cN{{\mathcal N}}

 \def\cR{{\mathcal R}}

 \def\cN{{\mathcal N}}

 \def\na{\nabla}

\def\one{\mathbf{1}}

 \newcommand{\bseq}{\begin{subequations}}
 \newcommand{\eseq}{\end{subequations}}

 \newcommand{\beq}{\begin{equation}}
 \newcommand{\eeq}{\end{equation}}
  \newcommand{\bal}{\begin{aligned} }
  \newcommand{\eal}{\end{aligned}}
  \newcommand{\bit}{\begin{itemize} }
  \newcommand{\eit}{\end{itemize}}
    \newcommand{\bga}{ \begin{gathered} }
  \newcommand{\ega}{ \end{gathered} }
 \newcommand{\ben}{\begin{eqnarray}}
 \newcommand{\een}{\end{eqnarray}}
 \newcommand{\beno}{\begin{eqnarray*}}
 \newcommand{\eeno}{\end{eqnarray*}}

 \newcommand{\uu}{\mathbf{u}}

 \newcommand{\xx}{\mathbf{x}}
 \newcommand{\yy}{\mathbf{y}}
\newcommand{\zz}{\mathbf{z}}

 \newcommand{\bbb}{\mathbf{b}}

 \newcommand{\R}{\mathbb{R}}

 \newcommand{\Id}{\mathrm{Id}}

\author{Jiajie Chen}
\thanks{Department of Mathematics, University of Chicago, Chicago, IL 60637. Email: jiajiechen@uchicago.edu}
\author{Thomas Y. Hou}
\thanks{Applied and Computational Mathematics, Caltech, Pasadena, CA 91125. Email:hou@cms.caltech.edu.}
 \date{}

\title[Analytic finite-rank corrections in 3D Euler Blowup]{Analytic finite-rank corrections 
for singularly weighted estimates in a computer-assisted proof of 3D Euler singularity}
\begin{document}
\maketitle

\begin{abstract}
Computer-assisted proofs of self-similar singularity formation for fluid equations often rely on numerically constructed approximate profiles. One effective approach to establishing stability of perturbations around a numerically constructed profile is to perform weighted energy estimates 
with singular weights near the singularity. 
However, the weighted norms require exact local vanishing conditions that are not automatically preserved by the equations nor the numerical construction. 
In this paper, we review an analytic low-rank correction method 
first developed in \cite{ChenHou2023a,ChenHou2023b} to overcome this difficulty. The numerical step determines coefficients, rigorous bounds, and low-order defect modes in explicit global basis representations, 
while the required vanishing conditions are enforced analytically through low-rank corrections derived from Taylor expansions of the relevant quantities represented in a smooth basis. For completeness, we briefly review the singularly weighted estimates and a quantitative finite-rank perturbation method 
in the 2D Boussinesq / 3D Euler stability argument, where singular weights and the required vanishing order arise. 
Against this background, we formulate the local correction principle in a simplified setting, explain the correction of the residual error in numerical constructions of approximate space-time solutions and the stream function, and discuss its broader applicability to computer-assisted stability analysis for nonlocal PDEs.

\end{abstract}

\section{Introduction}
\label{introduction}

Computer-assisted analysis has become an increasingly important tool in the
study of  nonlinear PDEs. In problems involving
self-similar or nearly self-similar blowup, however, the main challenge is not
merely to construct an approximate blowup profile numerically. 
The fundamental difficulty is to connect the numerical construction with a rigorous stability theory that is strong enough to control the residual error and the nonlinear perturbation around the approximate profile.

The present paper is motivated by the authors' work 
\cite{ChenHou2023a,ChenHou2023b} on computer-assisted proof of nearly
self-similar blowup for the 2D Boussinesq and the 3D axisymmetric
Euler equations with smooth data and boundary. 
We have established the following blowup result, summarized from \cite{ChenHou2023a,ChenHou2023b}.

  \begin{theorem}\label{thm:main}
There exists a family of smooth initial data with finite energy such that the solution of the 3D axisymmetric Euler equations  in the cylinder $r,z \in [0, 1] \times \mathbb{T}$ develops a nearly self-similar singularity in finite time, where $\mathbb{T}$ denotes the periodic torus.
  \end{theorem}

In \cite{ChenHou2023a,ChenHou2023b}, the approximate
self-similar profile is constructed numerically, the linearized operator is
decomposed into a leading operator that enjoys better stability properties analyzed by singularly weighted energy estimates plus a finite-rank operator. 
The finite-rank operator is estimated by constructing approximate space-time solutions with
rigorous error bounds. The overall argument combines singularly weighted
energy estimates, sharp nonlocal functional inequalities, 
quantitative finite-rank perturbation, and numerically constructed approximate solutions.

\subsection{Our computer-assisted proof strategy}\label{sec:strategy}
The 3D axisymmetric Euler equations near the singularity \cite{luo2014potentially} (located on the boundary away from the symmetry axis) can be formally approximated by the 2D Boussinesq 
equations \cite{chen2019finite2,majda2002vorticity,ChenHou2023a}:
\bseq\label{eq:bous}
\begin{align}
 \pa_t \om + \uu \cdot \na \om  &= \th_x ,   \label{eq:bous_a} \\
  \pa_t \th + \uu \cdot \na \th & = 0, 
\quad \uu  =  \na^{\perp}  (-\D)^{-1} \om, 
   \label{eq:bous_b} 
\end{align}
\eseq
in the upper half-space $(x, y) \in \mathbb{R}^2_+$ with  no-flow boundary condition $v(x, 0) = 0$, where $\mathbf{u} = (u, v)$ is the velocity, $\omega = u_y - v_x$ is the vorticity, and $\theta$ is the density or temperature. In this paper, we use boldface variables to denote vectors, e.g., $\xx = (x, y)$ and $\uu = (u, v)$. It is easy to show that
\[
\bal
  \td{\om}( \tau, \xx) &= C_{\om}(\tau) \om(t(\tau) ,  C_l(\tau) \xx ), \\
     \td{\th}(\tau,\xx) & = C^2_{\om}(\tau) C_l(\tau)^{-1}  \th( t(\tau), C_l(\tau) \xx ),  \\
\eal
\]
solve the dynamic rescaling equations 
 \beq\label{eq:bousdy0}
\bal
\td{\om}_{\tau} + ( c_l(\tau) \xx + \td{\uu} ) \cdot \na \td{\om}  &=   c_{\om}(\tau) \td{\om} + \td{\th}_x ,  \\
 \td{\th}_{\tau} + ( c_l(\tau) \xx + \td{\uu} ) \cdot \na \td{\th} &  = (c_l(\tau) + 2 c_{\om}(\tau) )  \td \th,
\eal
\eeq
where $\td \uu = (\td u, \td v)^T = \na^{\perp} (-\D)^{-1} \td{\om}$, $\xx = (x, y)^T$, $\tau$ is the rescaled time, $c_l(\tau), c_{\om}(\tau)$ are time-dependent scaling parameters, and 
\beq\label{eq:Cl_Cw}
\bal
  C_{\om}(\tau) = e^{ \int_0^{\tau} c_{\om} (s)  d s }, \quad
  C_l(\tau) = e^{\int_0^{\tau} -c_l(s) ds } ,
  \quad  t(\tau) = \int_0^{\tau} C_{\om}(s) d s .
\eal
\eeq

Since we mainly work on the dynamic rescaling equations \eqref{eq:bousdy0}, to simplify our presentation, we still use $t$ to denote the rescaled time $\tau$  and simplify $\td \om, \td \th$ as $\om, \th$. Then \eqref{eq:bousdy0} becomes: 
\beq\label{eq:bousdy1}
\bal
\om_t + (c_l \xx + \uu) \cdot \na \om & = \th_x + c_{\om} \om , \\
  \th_t + (c_l \xx + \uu)\cdot \na \th & =  
(c_l + 2 c_{\om}) \th .
\eal
\eeq

We impose normalization conditions on $c_{\om}, c_l$:
\beq\label{eq:normal}
c_l(t) = 2 \f{\th_{xx}(t,{\bf 0}) }{\om_x(t,{\bf 0})}, \quad c_{\om}(t) = \f{1}{2} c_l(t) + u_x(t,{\bf 0}) .
\eeq
For smooth data, these two conditions enforce
\beq\label{eq:normal1}
\theta_{xx}(t,{\bf 0})=\theta_{0,xx}({\bf 0}), \quad \omega_x(t,{\bf 0})=\omega_{0,x}({\bf 0}),
\eeq
for all time. This closes the system \eqref{eq:bousdy1}, \eqref{eq:normal}.

Our strategy is to first construct the approximate blowup profile $(\bar \omega, \bar \theta, \bar c_{\omega}, \bar c_l)$ for \eqref{eq:bousdy1} and prove that the perturbation
$\tilde{G} = (\omega(t) - \bar \omega,, \theta(t) - \bar \theta, c_{\omega}(t) - \bar c_{\omega}, c_l(t) - \bar c_l)$ remains small in some energy functional $E(\cdot)$: $ E ( \td G(t)) \ll 1$
for any $ t \geq 0 $. With this stability estimate and the values of the approximate scaling parameters $\bar c_l, \bar c_{\omega}$, the blowup asymptotics of the vorticity follows the form 
 \begin{eqnarray}
\omega(t,x) \approx \frac{1}{ T- t } \bar{\omega}(t, \frac{ x}{ (T-t)^{ \mu(t)}} ),  \qquad
 \| \tilde{\omega}(t, \cdot) -  \bar{\omega} \|_X \ll 1 \;,
 \end{eqnarray}
with some $\mu(t) > 0$ close to $- \f{\bar c_l}{\bar c_{\om}}\approx 2.92 $, where the dynamically rescaled vorticity profile $\tilde{\omega}(t, \cdot)$ from \eqref{eq:bousdy0} remains close to an (approximate) blowup profile $\bar \omega \neq 0$ in some norm $|| \cdot ||_X$ uniformly up to the blowup time $T$.
For the Hou-Luo blowup scenario \cite{luo2014potentially},  we prove $\mu(t) \approx 2.92$ in \cite{ChenHou2023a,ChenHou2023b}.

Our framework consists of three steps:

1. First construct the approximate blowup profile $(\bar \omega, \bar \theta, \bar c_{\omega}, \bar c_l)$ for \eqref{eq:bousdy1}. 

2. Linearize the equations around $(\bar \omega, \bar \theta, \bar c_{\omega}, \bar c_l)$ to obtain the perturbation equation for $\widetilde G = (\omega(t) - \bar \omega, \theta (t) - \bar \theta,  c_{\omega}(t) - \bar c_{\omega}, c_l(t) - \bar c_l)$: 
\beq
\label{eq:lin_nonlin} \partial_{t} \widetilde G = \mathcal{L} \widetilde G + \mathcal{N}(\widetilde G) + \mathcal{R}(\bar G), 
\eeq 
where $\mathcal{L}, \mathcal{N}, \mathcal{R}$ are the linear terms, nonlinear terms, and residual error, respectively.

3. 
Establish nonlinear stability estimates for a certain energy $E(t)$ of the perturbation $\widetilde G$, such as 
\bseq\label{eq:intro_EE}
\beq\label{eq:intro_EE:a}
 \frac{d}{d t} E(t) \leq -\lambda E(t) + C E(t)^2 + \varepsilon, \quad \lambda > 0, 
 \eeq
 where $-\lambda E$, $C E^2$, and $\varepsilon$ estimate the contributions from the linear term, nonlinear term, and the residual error of the approximate steady state, respectively.
Using a bootstrap argument, we obtain nonlinear stability $E(t) < E^*$ for some threshold $E^* \ll 1$ if $\varepsilon$ satisfies the \emph{quantitative bound}: 
\beq\label{eq:intro_EE_ep}
4 C \e  < \lam^2.
\eeq
\eseq

In \cite{ChenHou2023a}, the stability analysis is performed in weighted $L^{\infty}$-type spaces. 
We establish energy estimates similar to \eqref{eq:intro_EE:a}, and impose a smallness condition on the residual error similar to \eqref{eq:intro_EE_ep}. 
Compared to \eqref{eq:intro_EE},  
 the $L^{\infty}$-type stability estimates rely on a local condition on the residual error for each $\xx$, which is similar to but weaker than \eqref{eq:intro_EE_ep} and allows the residual error to be much larger away from the origin.  We refer to \cite[Lemmas A.1, A.2]{ChenHou2023a} for the localized 
 inequalities for stability and 
 discussions in \cite[Section 5.9]{ChenHou2023a} for the comparison between 
 the local and global stability conditions.

\subsection{Singularly weighted $L^{\infty}$-type estimates}\label{sec:intro_linf}

We derive the equations for the perturbation $(\om - \bar \om, \th - \bar \th)$ in Section \ref{sec:lin}. To obtain the crucial linear stability (leads to $\lam>0$ in \eqref{eq:intro_EE})  around an approximate self-similar profile, we must choose an appropriate functional space to control the growth of the perturbation $(\om - \bar \om, \th - \bar \th)$ for all time.  Below, we introduce the variables $\eta$ and $\xi$ to denote the perturbations of $\theta_x$ and $\theta_y$, respectively, and we abuse notation by letting $\om$ denote the perturbation in \eqref{eq:bousdy1}, rather than the unperturbed variable. In Section \ref{sec:lin}, we show that $\xi$ enjoys stronger damping than $\eta$ near the origin, 
and $\xi$ is weakly coupled with the $(\omega,\eta)$ subsystem near the boundary $y=0$. Thus,
we first focus on the coupling between $\omega$ and $\eta$ 
by setting $\xi = 0$ near the origin.  A first indication of the stability mechanism already appears if one keeps only the dominant local terms 
in \eqref{eq:lin} and freezes the coefficients near the origin. This leads to the following model system
\begin{equation}\label{eq:modelToy}
\omega_t+(a_1x\partial_x+a_2y\partial_y)\omega=-\omega+\eta,
\qquad
\eta_t+(a_1x\partial_x+a_2y\partial_y)\eta=a_3\eta,
\end{equation}
with parameters 
\beq\label{eq:model_a}
a_1\approx 0.5,
\qquad a_2\approx 5.5,
\qquad a_3\approx 0.5 ,
\eeq
which are the leading order coefficients of the local terms in the linearized equations \eqref{eq:lin}
\beq\label{eq:model_expand0}
\bar c_l x + \bar u = a_1 x + O(|\xx|^2),
\quad   \bar c_l y + \bar v = a_2 y + O(|\xx|^2) ,
\quad  2\bar c_\omega-\bar u_x = a_3 +  O(|\xx|) .
\eeq

We remark that, since $\na \bar\th$, $\bar\om$, $\bar u_y$, and $\bar v_x$ are $O(|\xx|)$ near $\xx=0$, the model \eqref{eq:modelToy} agrees with the linear system \eqref{eq:lin} at leading order near $\xx=0$, up to an $O(|\xx|)$ error, after dropping the nonlinear terms $N_i$, $N(\om)$, $N(\th)$ and the error terms $F_i$, $\overline F_{\om}$, $\overline F_{\th}$.
Thus, to analyze the linear stability of the full system \eqref{eq:lin}, one needs to obtain stability for the above system. 

Note that the coefficient $a_3>0$ means that the $\eta$ equation contains a growing local term. 

To extract the crucial damping terms, we use weighted $L^{\infty}$ estimates. By taking
\begin{equation}
\phi(x)=r^{-\gamma},
\qquad r=(x^2+y^2)^{1/2},
\label{eq:introWeightRadial}
\end{equation}
and computing the evolution of $\eta\phi$, one obtains
\begin{equation}
\partial_t(\eta\phi)+(a_1x\partial_x+a_2y\partial_y) (\eta\phi)=a(\gamma,x,y)\,\eta\phi,
\label{eq:introWeightedEta}
\end{equation}
where
\[
a(\gamma,x,y)=a_3+\frac{a_1x\partial_x\phi+a_2y\partial_y\phi}{\phi}.
\]
Since $a_2\ge a_1$, one has the upper bound
\begin{equation}
a(\gamma,x,y)\le a_3-a_1\gamma.
\label{eq:introDampingCoeff}
\end{equation}
Thus a singular weight creates damping: if $\gamma > \f{a_3}{a_1}$. 
Note that for the Boussinesq system, the ratio $\f{a_3}{a_1}$ is \emph{exactly} $1$.
\footnote{
From \eqref{eq:model_expand0} and Taylor expansion, we have $a_1 = \bar c_l 
+ \pa_x \bar u(0)$ and $a_3 = 2 \bar c_{\om} - \bar u_x(0)$. 
The normalization conditions \eqref{eq:normal} 
imply $ a_1 = a_3 $.}
The underlying stability mechanism is that, although there is a growth term $a_3 \eta$, the transport is strong enough to push the perturbation to spatial infinity, leading to linear stability 
in weighted $L^{\infty}$ norm. In the corresponding weighted $L^\infty$ estimate, one obtains inequalities of the form 
\footnote{
The integral version of estimate \eqref{eq:intro_wg} can be derived rigorously using the flow map associated with 
the flow  $(a_1 x, a_2 y)$ in \eqref{eq:modelToy}.
}
\begin{equation}
\frac{d}{dt}\|\omega\phi\|_{L^\infty}
\le (-1-a_1\gamma)\|\omega\phi\|_{L^\infty}+\|\eta\phi\|_{L^\infty},
\qquad
\frac{d}{dt} \|\eta\phi\|_{L^\infty}
\le (a_3-a_1\gamma)\|\eta\phi\|_{L^\infty}.
\label{eq:intro_wg}
\end{equation}
By taking $\g > \f{a_3}{a_1}$, we obtain exponential decay estimates 
for $\|\omega\phi\|_{L^\infty}, \|\eta\phi\|_{L^\infty}$. 
By taking a larger $\g$, we obtain stronger decay estimates.

Away from the origin, we use a crucial outgoing property of the flow \eqref{eq:outgoing} to propagate the stability of the local terms from the origin to the \emph{entire domain}; see Section \ref{sec:outgo}.

Since the Riesz transform $\pa_{ij}(-\D)^{-1}$ associated with the nonlocal velocity $\na \uu(\om)
= \na \na^{\perp} (-\D)^{-1}$ is not bounded in $L^{\infty}$, we perform the weighted $C^{1/2}$ estimates in \cite{ChenHou2023a} to close the estimates. We use ideas and methods similar to the above 
with singular weights to extract the damping terms in the weighted $C^{1/2}$ estimates.

By contrast, a weighted $L^2$ energy estimate suffers from the $y$-advection. 
If one takes a weight
$\phi=x^{-\alpha}y^{-\beta}, \ 0<\beta<1$, 
and performs weighted $L^2(\phi)$ estimate on the $\eta$ equation \eqref{eq:modelToy}, 
then integration by parts yields 
\begin{equation}
\f{1}{2} \f{d}{d t} \int \eta^2 \phi 
= \big( \frac{a_2(1-\beta)}{2} + \f{a_1(1-\al) }{2} + a_3  \big) \int \eta^2\phi,
\label{eq:introWeightedL2bad}
\end{equation}
where the term 
$\f{a_1(1-\al) }{2}, \f{ a_2(1-\beta)}{2}$ comes from the advection $a_1 x \pa_x, a_2 y \pa_y$ in \eqref{eq:model_a}, respectively. Since the energy must be finite and $\eta(x, 0) \neq 0$,
\footnote{
The Boussinesq equations \eqref{eq:lin} do not preserve the vanishing condition $\eta(x, 0)=0$ 
on the boundary $y=0$ 
due to the nonlocal terms.
}
one cannot choose $\beta\ge 1$.  
If $(1-\beta)$ is not small, since $a_2$ is much larger than $a_1, a_3$ (see \eqref{eq:model_a}), 
the term $ \f{ a_2(1-\beta)}{2}$ arising from the $y$-advection  in \eqref{eq:introWeightedL2bad} contributes a large positive growth term. On the other hand, if $(1-\beta)$ is small, 
the estimate of the nonlocal term $u_x$ in \eqref{eq:lin} contributes a large constant 
$(1-\beta)^{-1/2}$
\begin{equation}
\|u_x\phi^{1/2}\|_{L^2}
\le C (1-\beta)^{-1/2} \|\omega\phi^{1/2}\|_{L^2} .
\label{eq:introUxL2bad}
\end{equation}
As a result, one needs to take a very singular weight $x^{-\al} y^{-\b}$ 
or $  |(x,y)|^{-\al} y^{-\b}$ with large $\al$ to 
extract the desired damping effect in the singularly weighted $L^2$ estimates.

A potential approach is to perform $H^k$ estimate \cite{merle2022implosion2,buckmaster2022smooth,chen2024Euler} with large $k$, but the coercivity of the linearized operator in $H^k$ estimate is unclear, even for top-order terms.\footnote{
 Taking a partial derivative $\pa_x^i \pa_y^j$ plays a role similar to that of a singular weight $ x^{-i}y^{-j}$. However, to derive a damping term for $ \pa_x^k \eta$ in the model \eqref{eq:modelToy}, one needs $k \geq 14$, and the resulting equations contain many mixed derivatives, such as 
 $\pa_x^k \pa_y^l \eta$, and difficult commutator terms.
Due to the boundary, $\pa_y$ does not commute with the nonlocal operator $ \na^{\perp}(-\D)^{-1}$ for the velocity. Thus, in the $H^k$ estimates, it is not clear if one can obtain coercivity estimates for the leading order operator. 
}
Moreover, constructing a profile with a small residual in $H^k$ norm for large $k$, e.g. $k \geq 14$, would be extremely difficult. These considerations lead us 
to use singularly weighted $L^\infty$ and weighted $C^{1/2}$ estimates for stability estimates.

We remark that singularly weighted energy estimates provide a powerful approach to stability analysis for (nearly) self-similar blowup. The exponent $\b$ in the singular weight $|\xx|^{-\b}$ precisely captures the stability properties of the linearized operator near the origin. 
This method has proved robust in a range of settings: for numerically constructed blowup profiles in 
1D fluid models by Chen--Hou--Huang \cite{chen2019finite,chen2021HL};
for analytic blowup profiles in 1D fluid models by Chen \cite{chen2020singularity,chen2020slightly};
for Type-I self-similar blowup in the semilinear heat equation and the complex Ginzburg--Landau equation with logarithmic correction \cite{hou20242,chen2025stability}, using purely energy estimates without delicate spectral analysis; for establishing \emph{full} stability of the angular velocity, a key ingredient in proving vorticity blowup in compressible Euler equations \cite{chen2024Euler,chen2024vorticity}; and for the \emph{exact} characterization of unstable modes in smooth implosion for \emph{non-isentropic} compressible Euler equations \cite{chen2026implosion}.

\subsection{Quantitative Finite-rank perturbation} 
\label{sec:finite_rank}

The most challenging part in the linear stability analysis is to control the nonlocal velocity  $\na \uu(\om) = \na \na^{\perp}(-\D)^{-1} \om$; see \eqref{eq:lin}.
Standard Calder\'on--Zygmund theory does not provide constants that are sufficiently explicit or sharp 
enough so that we can treat $\uu$ perturbatively. To estimate the nonlocal terms, we derive sharp $C^{1/2}$ estimates for $\na \uu$ using the symmetry properties of the kernels and some techniques from optimal transport \cite{villani2021topics}. We decompose the Biot-Savart law
into two parts. The main part captures the most singular part of the Biot-Savart law, and we apply the sharp functional inequalities for its $C^{1/2}$ estimate. 
The other part is more regular, and we approximate it by a finite rank operator and obtain sharp estimates by constructing approximate space-time solutions with rigorous error control.
For more discussion,
we refer to Section  \ref{sec:nonlocal} 
 and focus on the finite-rank perturbation below.

We decompose the solution $W = (\omega, \theta_x, \theta_y)$ 
into the approximate profile  $\overline{W} = (\bar{\omega},\bar{\theta}_x, \bar{\theta}_y)$
and the perturbation $\td {W}$, and denote by $\cL$ the linearized operator around $\overline{W}$ that governs the perturbation $\td {W}$ in the dynamic rescaling formulation (see Section \ref{sec:lin}):
\beq\label{eq:intro_lin}
\partial_t \td W= \mathcal{L} \td W .
\eeq
We decompose the perturbation $ \td W=\td W_1+ \td W_2$ and the operator $ \cL = \cL_0 + \cK$ 
for some finite-rank operator $\cK$ to be chosen,  and decompose the above equation as 
\begin{equation}\label{eq:intro_split}
\partial_t \td W_1={\mathcal L}_0 \td  W_1,\qquad
\partial_t \td W_2={\mathcal L} \td W_2+ \cK \td W_1.
\end{equation}
We design the finite rank operator $\cK$ to approximate more regular nonlocal terms 
in the full operator $\cL$, so that after subtracting $\cK$ from $\cL$, the nonlocal parts in the 
resulting operator $\mathcal L_0$ can be treated perturbatively to the local damping terms (derived similarly to that in Section \ref{sec:intro_linf}).   
We obtain linear stability estimates for $W_1$ using singularly weighted energy estimates.

The effect of  $\cK$  is estimated through approximate space-time solutions.
We explain this idea by assuming that $\cK$ is a rank-one operator 
\[
\cK(W)=a(x)P(W), \quad 
P(W)=\int \bar p(x)W(x)\,dx ,
\]
for some known coefficient $a(x)$ and a bounded scalar functional $P(W)$.  In this case, $W_2$ in \eqref{eq:intro_split} is  driven by the rank-one forcing and can be  represented through the Duhamel
formula
\begin{equation}
W_2(t)=\int_0^t P(W_1(s))\,e^{{\mathcal L}(t-s)}a(x)\,ds.
\label{eq:introDuhamel-1}
\end{equation}

Since $\cL$ is complicated, 
we cannot estimate $e^{\cL(t-s)} a(x)$ directly. 
Instead, we construct a numerical 
approximation $\wh g(t, \xx)$ for $e^{\cL (t- s)} a(x)$ by solving the linear equation 
$  \pa_t g = \cL g $ from $g(0, \xx) = a$ numerically. To track the error in solving $\wh g$ rigorously, we introduce the residual operator $\cR(W_1)$, 
\beq\label{eq:cR_intro}
\cR(\td  W_1) \teq P( \td W_1(t))  \cdot ( \wh g(0, \xx) - a(\xx) ) + \int_0^t  P( \td W_1(s)) \cdot  (\pa_t -\cL) \wh g( t- s)  ds  .
\eeq

With the approximate solution $\wh g$, we modify the formula for $W_2$ as 
\beq\label{eq:W2_form_hat_intro}
 \wh W_2 = \int_0^t   P(W_1(s)) \cdot \wh g(t- s, \xx)  d s.
\eeq

We also modify the decomposition \eqref{eq:intro_split} as follows
\bseq\label{eq:decomp_m}
\begin{align}
\pa_t \td  W_1  &=  \cL_0 \td  W_1 - \cR(   \td W_1, t) , \label{eq:decomp_ma} \\
\pa_t \wh W_2 &= \cL \wh W_2 + \cK  \td W_1 +  \cR(  \td W_1 , t) , \label{eq:decomp_mb}
\end{align}
\eseq
so that $\td W_1 + \wh W_2$ solves \eqref{eq:intro_lin}, where $\td W_1|_{t=0} = \td W_0, \; \wh W_2|_{t=0}=0$. Note that $\wh W_2$ solves \eqref{eq:decomp_mb} \textit{exactly} since the residual operator captures \textit{all} the errors. 
If the error $( \pa_s - \cL) \wh g(s, \xx)$ and $\wh g(0, \xx) - a(\xx) $ are small, 
we can treat the residual operator $\cR$ as a small perturbation
to the stable part $\cL_0$
in some suitable functional space using a bootstrap argument:  $|| \cR( W_1, t) ||_X \leq \hat \e \sup_{s\leq t} || W_1(s)||_X, \  \hat \e \ll 1$. 

We remark that we \emph{only} perform the singularly weighted estimate on $\td W_1$ 
and $\wh W_2$ is estimated via the formula \eqref{eq:W2_form_hat_intro}.
By proving that $\cL_0$ is stable in some functional spaces $X$ and 
verifying the temporal decay of the numerical function $\wh g$, we obtain the stability of the full perturbation $W = \td W_1 + \wh W_2$. We generalize the  above method to finite-rank operator $\cK$ using linearity, 
and develop a \emph{quantitative finite-rank perturbation} method.

We note that decomposing the linearized operator into a coercive operator 
plus a compact or finite-rank operator $\cK$ 
and the splitting \eqref{eq:intro_split} are powerful analytical approaches. The above approach obtains \emph{quantitative} estimates of the 
compact perturbation with \emph{computable constants}, and is well suited for computer-assisted proofs
which require \emph{computable} constants. 
The method is very different from compact perturbation via semigroup methods, see e.g. \cite{merle2022implosion2,buckmaster2022smooth,chen2024vorticity,chen2024Euler,jia2015incompressible}, which generally yield estimates with \emph{non-computable} constants. This 
\emph{quantitative compact perturbation} method has been used in an essential way in the recent work of Hou-Wang-Yang to establish a computer-assisted proof of nonuniqueness of Leray--Hopf solutions to the unforced 3D Navier--Stokes equations \cite{HouWangYang2025}. 
The splitting \eqref{eq:intro_split} and the decomposition of the perturbation $\td W$ into a \emph{fully} stable component $\td W_1$ and a more regular component $\td W_2$ estimated via the Duhamel formula 
\eqref{eq:intro_split}, \eqref{eq:introDuhamel-1} have been extended in subsequent works by Chen and collaborators to singularity formation in compressible Euler equations \cite{chen2024vorticity,chen2024Euler} and to the Landau equation in kinetic theory \cite{bedrossian2026finite}.

\subsection{The need for analytic low rank corrections}\label{sec:analytic_low_rank}

In the Boussinesq/Euler stability analysis, the singular weights are effectively of order $|\xx|^{-3}$ near the origin. 
To employ the singularly weighted energy estimates, the perturbation must vanish cubically near the origin 
\begin{equation}
f(\xx)=O(|\xx|^3) ,
\qquad \text{near } \xx=0 ,
\label{eq:introCubicNeed}
\end{equation}
so that the weighted $L^{\infty}$ energy norm is bounded.  However, the natural odd/even symmetry class of the perturbation 
and the equations preserves only quadratic vanishing order (see discussion in Section \ref{sec:lin}). 
Thus, one has to perform low-rank correction to the perturbation. 

Using the finite-rank perturbation method in Section \ref{sec:finite_rank} and 
choosing low-rank correction modes to $\cL$ near the origin based on Taylor expansion
\footnote{
This finite collection of modes is part of the finite-rank operator $\cK$.
}, 
the resulting operator $\cL_0 W_1 = (\cL - \cK) W_1$ preserves cubic vanishing order. This implies that $\td W_1$ in \eqref{eq:intro_split} satisfies $\td W_1(t,\xx) = O(|\xx|^3)$ near $\xx=0$, and thus belongs to the singularly weighted space.

However, for the modified system \eqref{eq:decomp_m}, this correction is more delicate since the approximate profile, the approximate space-time solutions $\wh g$, and the numerical stream function are all constructed numerically. If the coefficients in \eqref{eq:decomp_m} come from numerical computations, 
due to numerical error, how can one rigorously justify an \emph{exact} cubic vanishing condition after correction?  

We develop the analytic low-rank correction method in \cite{ChenHou2023a,ChenHou2023b} precisely 
to address this difficulty and construct a \emph{corrected} numerical solution with the desired vanishing order near the origin. The key point is that the numerical step \emph{only} determines coefficients in explicit basis representations. Once these coefficients are fixed, they determine functions defined on the entire domain, rather than merely numerical
values on grid points. The subsequent low-rank corrections are performed \emph{analytically} on the resulting globally defined functions, using their Taylor expansions near the origin.  Hence the improved vanishing order is an \emph{exact} analytic property of the corrected function.

In the stability analysis, we apply analytic low-rank  corrections in two parallel forms. The first arises in the construction of the approximate space-time solution, e.g. $\wh g$ in \eqref{eq:cR_intro},
\eqref{eq:W2_form_hat_intro}, appearing in the Duhamel representation. 
To ensure that  \eqref{eq:decomp_ma} preserves the cubic vanishing order, we require that 
the errors $ \wh g(0, \xx) - a(\xx), (\pa_s -\cL) \wh g(s)$ in \eqref{eq:cR_intro} vanish $O(|\xx|^3)$ near $\xx=0$. Due to numerical error, the raw numerical construction does not satisfy these vanishing conditions. We therefore perform two analytic low-rank corrections near the origin to obtain the desired vanishing order. 
As a result, the residual operator $\cR$ and the modified equation \eqref{eq:decomp_ma} preserve the cubic vanishing order $W_1(t, \xx) = O(|\xx|^3)$, and $W_1$ belongs to the singularly weighted space.

The second arises in the numerical construction of the stream function. The 
residual error in solving the Poisson equation for the 
uncorrected  stream function fails to satisfy the exact vanishing condition required by the singularly weighted norm (see discussion in Step 2 in Section \ref{sec:numer_vel}). One therefore performs an analogous finite-rank correction to the stream function, designing the correction so that the induced velocity can be derived explicitly and analytically. As a result, the corrected residual and the induced velocity error belong to the admissible singularly weighted space.

The purpose of this paper is to isolate this analytic low-rank correction method and present it in a self-contained way. We do not attempt to reproduce the entire stability proof nor review the main steps of the proof in 
\cite{ChenHou2023a,ChenHou2023b}. 
We emphasize that this correction method is not merely an auxiliary construction, but 
an essential part of the linear stability argument. Without it, the 
numerical solution or its induced error fail to satisfy the required vanishing conditions and  become unbounded in the singularly weighted norm. The analytic low-rank correction method, together with
the finite-rank perturbation method, 
upgrades the vanishing order of the perturbation from $\td W = O(|\xx|^2)$ to exact cubic vanishing order for $\td W_1$,    while preserving the explicit basis representation (such as those in $\wh g, \wh W_2$ in \eqref{eq:W2_form_hat_intro}) needed for rigorous verification. 
This is precisely why the method deserves to be presented separately and in some detail.

\subsection{A very brief literature review}

We do not attempt to survey the extensive literature on Euler singularities and related model problems, and instead mention only the works most directly relevant to our discussion.

Elgindi and collaborators proved the remarkable result that the  axisymmetric
Euler equations in $\R^3$ without swirl can develop a finite-time singularity for 
$C^{1,\alpha}$ initial velocity with sufficiently small $\alpha$
\cite{elgindi2019finite,elgindi2019stability}.  Subsequently, Chen-Hou \cite{chen2019finite2} established finite-time singularity in the 2D Boussinesq
and 3D Euler equations with $C^{1,\al}$ velocity and boundary. 
In these works, the parameter $\alpha$ plays an essential role as a small parameter.

After our works \cite{ChenHou2023a,ChenHou2023b} on singularity formation in 
2D Boussinesq and 3D Euler equations with smooth data and boundary, there have been several new developments.
Cordoba, Martinez-Zoroa, and Zheng \cite{cordoba2023finite} developed a new framework based on layer constructions and established blowup for 3D axisymmetric Euler equations with $C^{1,\alpha} \cap C^{\infty}(\R^3 \backslash \{0 \})$ velocity and small $\al$. In \cite{chen2024remarks}, Chen refined the construction of initial perturbation in \cite{elgindi2019finite}, \cite{chen2019finite2} to obtain self-similar blowup solution with $C^{\infty}$ regularity except the blowup point. 
Elgindi-Pasqualotto  \cite{elgindi2023instability} constructed finite time blowup of 
 Boussinesq equations in $\R^2$ and 3D axisymmetric Euler equations away from the 
axis with $C^{1,\al}$ velocity and small $\al$.
Cordoba and Martinez-Zoroa \cite{cordoba2023blow}  established blowup for the forced 3D Euler equations with smooth velocity and a bounded, solution-dependent force $f \in C^{1,1/2-}$.

Very recently, Shkoller \cite{shkollar2026} developed a new Lagrangian framework to prove finite-time blowup for the 3D axisymmetric Euler equations without swirl from finite-energy $C^{1,\alpha}$
 initial velocity for every $\alpha \in (0,1/3)$. Independently, Chen  \cite{chen2026eulerI,chen2026eulerII} 
constructed exact $C^{1,\alpha}$
 self-similar profiles and proved asymptotically self-similar blowup for 3D axisymmetric Euler equations without swirl from compactly supported $C^{\alpha}$ initial vorticity and finite-energy $C^{1,\al}$ initial velocity for every $\alpha \in (0,1/3)$, by lifting singularities of a 1D model to the 3D Euler equations. Also independently, Shao, Wei, P. Zhang, and Z. Zhang \cite{shao2026self} developed 
a Schauder fixed-point framework to construct exact self-similar blowup profiles 
for 3D axisymmetric Euler equations without swirl with $C^{1,\alpha}$ velocity for every $\alpha \in (0,1/3)$. These three works use distinct methods and establish blowup for the full range
$\alpha\in(0,\tfrac13)$. We remark that the 3D axisymmetric Euler equations 
without swirl admits global solutions from 
finite-energy $C^{1,\alpha}$ initial velocity for all $\alpha\geq\tfrac13$, including the endpoint  $\alpha=\tfrac13$ \cite{shao2026global}.

\vspace{0.1in}
\paragraph{\bf Organization}

The paper is organized as follows. In Section \ref{sec:review}, we briefly review
the stability analysis in \cite{ChenHou2023a,ChenHou2023b}, including the singularly weighted estimates and nonlocal estimates. 
Starting from Section \ref{sec:principle}, we focus on the analytic low-rank correction method. We explain the principle of 
numerical construction versus analytic correction, discuss the correction of the approximate space-time solution and its induced residual error in a model problem, and  describe similar correction for the numerical stream-function solver.  Our goal is to keep the overall structure of the full stability analysis visible while providing a self-contained explanation of the analytic low-rank corrections that  rigorously reconcile the singularly weighted framework with numerical constructions.

\section{Brief review of the stability analysis in \cite{ChenHou2023a}}
\label{sec:review}

In this section, we briefly review the stability analysis in \cite{ChenHou2023a} before turning to the specific low-rank correction method.

\subsection{The 2D Boussinesq/Euler stability setting}
\label{sec:lin}

Recall from Section \ref{introduction} that the dynamically rescaled 2D Boussinesq system is given by:
\begin{equation}
\pa_t \omega+(c_lx+u)\cdot\nabla\omega=\theta_x+c_\omega\omega,
\qquad
\pa_t \theta+(c_lx+u)\cdot\nabla\theta=c_\theta\theta,
\label{eq:rescaledBq}
\end{equation}
where $c_\theta = c_l + 2 c_{\om}$, the velocity $\uu$ is recovered from $\omega$ using the Biot--Savart law. 
The rescaling parameters $c_l, c_{\om}$ are determined via the normalization conditions \eqref{eq:normal}.
We first construct the approximate profile  with sufficiently small residual error  by solving \eqref{eq:rescaledBq} numerically for a long enough time.
A stable nearly self-similar singularity is then reduced to the nonlinear stability of an approximate steady state  $(\bar\omega,\bar\theta,\bar \uu,\bar c_l,\bar c_\omega)$ of the  system \eqref{eq:rescaledBq}.

We linearize \eqref{eq:rescaledBq}  around the approximate steady state and denote the perturbation as $\tilde{\omega} = \omega - \bar{\omega}$, $\tilde{\uu} = \uu - \bar {\uu}$, and $\tilde{\theta} = \theta - \bar{\theta}$. To simplify the notations, we still use $(\omega, \uu, \theta)$ to denote the perturbation $(\tilde \omega, \tilde u, \tilde \theta)$. We obtain the following linearized equations for the perturbation $(\omega,\theta)$:
\bseq\label{eq:lin}
\begin{align}
\pa_t  \omega & = -(\bar c_lx+\bar \uu)\cdot\nabla\omega+\theta_x+\bar c_\omega\omega- \uu\cdot\nabla\bar\omega+c_\omega\bar\omega+\overline F_\omega+N(\omega), \label{eq:linomega}  \\
\pa_t \theta & = -(\bar c_lx+\bar \uu)\cdot\nabla\theta+\bar c_\theta\theta+c_\theta\bar\theta- \uu\cdot\nabla\bar\theta+\overline F_\theta+N(\theta), \label{eq:lintheta}
\end{align}
where $\uu= \na^{\perp}(-\D)^{-1} \om$, $\overline F_\omega,\overline F_\theta$ are residual errors, and $N(\omega),N(\theta)$ denote nonlinear terms. \footnote{
Since the analysis of the nonlinear terms and error terms are not the main focus of this paper,
we do not present their formulas and refer them to \cite[Section 2.5]{ChenHou2023a}. 
}

Taking derivatives on \eqref{eq:lintheta} and denoting 
$\eta=\theta_x, \xi=\theta_y$, we obtain 
\begin{align}
\partial_t\eta & = -(\bar c_l \xx+\bar \uu)\cdot\nabla\eta + (2\bar c_\omega-\bar u_x)\eta - \bar v_x\xi - u_x\cdot\nabla\bar\theta - \uu \cdot\nabla\bar\theta_x + 2c_\omega\bar\theta_x + N_2 + F_2,
\label{eq:etafull}\\
\partial_t\xi & = -(\bar c_l \xx+\bar \uu)\cdot\nabla\xi + (2\bar c_\omega+\bar u_x)\xi - \bar u_y\eta - u_y\cdot\nabla\bar\theta - \uu\cdot\nabla\bar\theta_y + 2c_\omega\bar\theta_y + N_3 + F_3.
\label{eq:xifull}
\end{align}
\eseq
Because of the normalization conditions $\pa_{xx} \th(t, {\bf 0} ) = 0, \pa_x \om(t, {\bf 0} ) = 0$ (derived from \eqref{eq:normal1} for the full solution), the odd symmetry of $\om(x, y)$ in $x$,
and the even symmetry of $\th(x, y)$ in $x$, the perturbation satisfies the local vanishing conditions
\begin{equation}
\omega=O(|x|^2),
\qquad
\theta_x=O(|x|^2),
\qquad
\theta_y=O(|x|^2)
\qquad\text{near }x=0.
\label{eq:quadratic-vanishing}
\end{equation}
Thus the natural symmetry class yields only quadratic vanishing order.

In the $\xi$-equation \eqref{eq:xifull}, coefficient $2\bar c_\omega+\bar u_x$ is strongly negative near the origin, whereas in the $\eta$-equation \eqref{eq:etafull}, coefficient $2\bar c_\omega-\bar u_x$ is slightly positive. In fact, we may explicitly derive 
$2\bar c_\omega+\bar u_x \approx -5.5, \ 2 \bar c_{\om} - \bar u_x \approx 0.5$,
near $\xx=0$. Thus $\xi$ enjoys stronger damping than $\eta$. Moreover, since $\pa_x \bar v(x,0) =0$, 
$\xi$ is weakly coupled with $(\om, \eta)$ subsystem \eqref{eq:lin} near the boundary.

\subsection{Steps of blowup analysis}

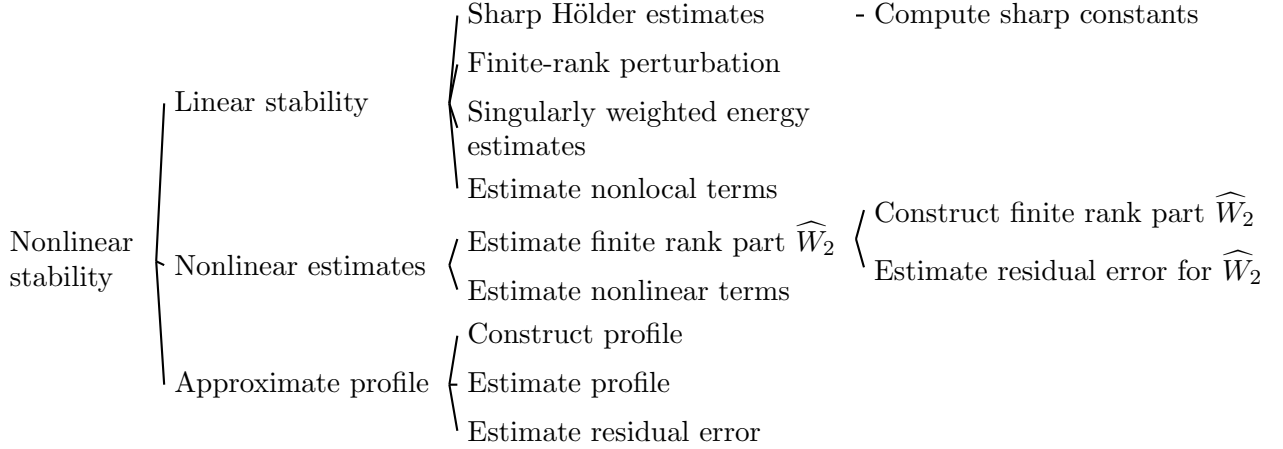
\begin{figure}
  \centering
\begin{forest}
 for tree={l sep=3pt,
    s sep = 1pt,
    grow=east,
    edge={thick},
    parent anchor=east,
    child anchor=west,
    }
      [ 
      Nonlinear \\ stability
        ,  text width= 1.8cm 
              [
              Approximate
                profile  , text width=3.5cm                                                                          
                    [Estimate residual error , 
                    text width=4cm
                    ]
                    [Estimate profile 
                     , text width=4cm]    
                    [Construct profile, text width=4cm ]  
              ]
              [Nonlinear  estimates,  text width=3.5cm 
                  [Estimate  nonlinear  terms 
                    , text width=5cm
                   ]
                  [Estimate  finite   rank 
                    part $\wh W_2$ 
                    , text width=5cm      
                       [Estimate residual  
                       error for $\wh W_2$             
                        , text width=5.5cm]                            
                        [Construct finite  rank  part $\wh W_2$, 
                        text width= 5.5cm] 
                  ]
              ]     
              [Linear 
              stability,   text width=3.5cm    
                  [Estimate   nonlocal terms , text width=5cm
                  ]
                  [Singularly weighted energy  estimates,  text width=5cm]
                  [Finite-rank  perturbation,  text width=5cm]
                  [Sharp H\"older estimates , text width= 5cm
                      [Compute sharp   constants 
                      , text width= 4.5cm]
                  ]
              ]                                 
            ]
      ]
\end{forest}

\caption{Level 2: main steps. Level 3: main estimates and methods. Level 4: compute and estimate explicit functions or constants.
}
\label{fig:tree}

\end{figure}

The nonlinear stability estimates of the approximate blowup profile rely on three main components in the dependency tree in Figure \ref{fig:tree}. The most difficult part is establishing linear stability, corresponding to establishing $\lambda > 0$ in \eqref{eq:intro_EE:a}.
By constructing an approximate profile with sufficiently small residual error ($\e \ll 1$ in \eqref{eq:intro_EE}), we treat the nonlinear term and residual error perturbatively (corresponding to
$C E(t)^2$ and  $\e$ in \eqref{eq:intro_EE}).

We do not attempt to reproduce the full stability proof or review the main steps  in 
\cite{ChenHou2023a,ChenHou2023b}, 
but instead explain  why singular weights, weighted $L^\infty$ and $C^{1/2}$ spaces, optimal transport, and finite-rank perturbations are all needed for establishing linear stability. 
Then we formulate the analytic low-rank correction, which ensures that the numerical constructions and their induced errors satisfy the required vanishing conditions and lie in the appropriate weighted spaces.
For a review of the main steps in the proof \cite{ChenHou2023a,ChenHou2023b}, we refer to 
\cite{chen2025singularity}.

\subsection{Singularly weighted estimates and outgoing condition}
\label{sec:outgo}

In Section \ref{sec:intro_linf}, we have  discussed the weighted $L^{\infty}$ estimate 
in the model \eqref{eq:modelToy}, whose coefficients approximate the $(\om, \eta)$ subsystem \eqref{eq:modelToy} near $ \xx =0 $ to the leading order.  We have also explained the advantages of employing weighted $L^\infty$-type spaces over weighted $L^2$ or $H^k$ spaces for the stability analysis. To avoid redundancy, we do not repeat this discussion below. The weighted $L^{\infty}$ stability analysis 
of $\xi$ in \eqref{eq:xifull} near $\xx=0$ is obtained similarly.

To extend the stability analysis of the local parts near $x=0$ in Section \ref{sec:intro_linf} to the linear operator in \eqref{eq:lin} in $\R_2^+$, we use a crucial outgoing property of the profile  identified in  \cite[Section 2.4]{ChenHou2023a}
\beq\label{eq:outgoing}
 \bar c_l x + \bar u(x, y) \geq c_1 x ,  \quad 
  \bar c_l y + \bar v(x, y) \geq c_2 y , \quad \forall \, x, y \geq 0.
\eeq

Below, we show that the property \eqref{eq:outgoing} can generate \textit{arbitrary strong} damping terms in $|x| \in [R_1, R_2]$ for any $0< R_1< R_2$. Denote $W = (\om, \eta, \xi)$. The linearized equations \eqref{eq:lin} without nonlinear terms and error terms can be written  schematically as 
\beq\label{eq:full_scheme_sys}
  \pa_t W + \bbb(\xx) \cdot \na W = \overline M \cdot W + \cL_{nloc} \om,\quad 
  \bbb(\xx) = \bar c_l \xx + \bar \uu ,
\eeq
where $\overline M(\xx) \in \R^{3 \times 3}$ denotes the coefficient matrix for the local terms, e.g. 
$\bar c_{\om} \om, (2 \bar c_{\om} - \bar u_x) \eta$, and $\cL_{nloc} \om$ denotes the nonlocal terms, including the $c_{\om}$ terms, which depend only on $\om$.

Multiplying both sides by a weight $\vp$ and commuting $\vp$ with the transport term, we obtain 
\beq\label{eq:damp_sys}
  \pa_t (W \vp) + \bbb(x) \cdot \na ( W \vp ) = (\overline M + d_{\vp}(x) \Id ) \cdot (W \vp)  
  + \vp \cL_{nloc} \om , \quad 
    d_{\vp}(x) \teq (\bbb \cdot \na \vp) \vp^{-1},
\eeq
where $d_{\vp}(x) \cdot W \vp$ comes from the commutator. We extract damping terms from $d_{\vp}(x)$ and design 
\beq\label{eq:design_wg}
  \vp = \vp_s \vp_0^m , \quad \vp_s = |\xx|^{-\al} + |\xx|^{\b}, 
  \quad \vp_0 \asymp 1 ,
\eeq
with $\al > 0, m > 0$. 
A direct calculation yields 
\beq\label{eq:lin_wg}
  d_{\vp} = ( \bbb \cdot \na \vp_s ) \vp_s^{-1}
  +  m ( \bbb \cdot \na \vp_0) \vp_0^{-1} \teq d_1( \vp_s ) + m d_2(\vp_0). 
\eeq

From \eqref{eq:lin_wg}, the damping terms arising from the weights can be decomposed into 
the asymptotic part $d_1(\vp_s)$, which depends on the exponents $\al$ and $ \b$,
and the term $m d_2(\vp_0)$, which governs the $O(1)$ region away from $0$. The outgoing condition \eqref{eq:outgoing} generates a strong stability mechanism, quantified as follows.

\begin{lemma}[\bf Arbitrary strong damping]\label{lem:damp}
Let $d_2(\vp_0)$ be the coefficient defined in \eqref{eq:lin_wg}. For any $0 < R_1 < R_2, D > 0$, there exists a smooth weight $ \vp_0\asymp_{R_1, R_2} 1$ depending on $R_1, R_2$ and $ m = m(R_1, R_2, D) >0$ depending on $R_1, R_2, D$ such that 
\[
m d_2(\vp_0) \leq  - D \one_{ R_1 \leq |x| \leq R_2} .
\]
\end{lemma}

A similar lemma is presented in our review paper \cite[Lemma 2]{chen2025singularity}.

\begin{proof}
There are many ways to design $\vp_0$. We consider a radial weight $\vp_0( \xx) = \vp_0( |\xx|)$ with 
\[
  \vp_0(\xx)= 1, \  |\xx| \leq R_1/2, \quad \vp_0(\xx) = 1/2 ,  \ |\xx| \geq 2 R_2, 
  \quad \pa_r \vp_0 \leq 0 \  \forall   \xx \in \R_2^+, 
  \quad \pa_r \vp_0 < 0,  \ |\xx| \in [R_1, R_2] ,
\]
where $\pa_r$ denotes the radial derivative. Then we obtain $ \pa_x \vp_0(\xx) \leq 0, \pa_y \vp_0(\xx) \leq 0$.  Suppose $ \pa_r \vp_0 \leq - c(R_1, R_2) \one_{[R_1, R_2]}( |\xx| )$. Using \eqref{eq:outgoing}, we obtain
\[
 ( \bbb \cdot \na \vp_0 ) \vp_0^{-1} 
 \leq ( c_1 x \pa_x \vp_0 + c_2 y \pa_y \vp_0 ) \vp_0^{-1}
 \leq \min( c_1, c_2 ) ( r \pa_r \vp_0( \xx) ) \vp_0^{-1} \leq  - C(R_1, R_2)  \one_{[R_1, R_2]}( |\xx| ),
\]
with some constant $C(R_1, R_2)> 0$. Taking $m$ sufficiently large, we complete the proof.
\end{proof}

The effect of the damping terms $d_1(\vp_s)$ near $\xx = 0$ follows from the discussion in Section \ref{sec:intro_linf}. 
In particular, by choosing $\al > \al_*$ for a threshold $\al_*$ determined by the coefficients 
near $0$ , we obtain a damping term near $\xx=0$.
For $\xx$ large enough, due to the decay of the coefficients $ |\na \bar \uu |, |\na \bar \th|, 
\bar \om, |\bar \om| \to 0$, the system \eqref{eq:lin} is effectively governed by 
a system similar to \eqref{eq:modelToy}. Using an argument similar to Section \ref{sec:intro_linf}, 
we determine the range of $\b$ that generates damping terms.  

Thus, combining the effects of $d_1(\vp_s)$ and $md_2(\vp_0)$ from Lemma \ref{lem:damp}, we obtain a coefficient matrix 
$\mathsf{H} = \overline M + d_{\vp}$ whose diagonal entries are negative and which is weighted diagonally dominant: 
\beq\label{eq:diag_domainte}
   - (\overline M_{ii} + d_{\vp,i} ) - 
   \sum\nolimits_{j\ne i} |\overline M_{ij}|  \mu_i \mu_j^{-1} \geq \lam ,
\eeq
for some $\mu_1, \mu_2, \mu_3 > 0$. The above condition implies exponential decay 
of the energy $\max_i( \mu_i  \| W_i \vp \|_{L^{\infty}} )$ 
for the local system of \eqref{eq:full_scheme_sys} (See \cite[Lemma A.1]{ChenHou2023a})
\[
  \pa_t W + \bbb(\xx) \cdot \na W = \overline M \cdot W.
\]

In \cite{ChenHou2023a}, since the variables $\om$ and $\eta, \xi$ have different far-field asymptotics, we design $\al, \b$ in \eqref{eq:design_wg} and different weights $\vp_i$
for $\om, \eta, \xi$ to capture these asymptotics, and achieve the diagonally dominated condition 
\eqref{eq:diag_domainte} for linear stability.  At the top order, we perform the weighted $C^{1/2}$ estimates for $\om \psi_1, \eta \psi_2, \xi \psi_3$ in \cite{ChenHou2023a} to close the stability estimates, as motivated in Section \ref{sec:intro_linf}. The main part of the energy 
(corresponding to $E$ in \eqref{eq:intro_EE}) consists of 
\beq\label{eq:EE_top}
E_M(f) =\max( \max_{i\leq 3} || f_i \vp_i ||_{L^{\infty}}, 
\max_{i\leq 3}|| f_i \psi_i ||_{C_{g_i}^{1/2}} ),  \quad
 \| F \|_{C_{g_i}^{1/2} }
\teq || (F(\xx) - F(\yy)) g_i(\xx-\yy) ||_{L^{\infty}(\R^2_+ \times \R^2_+)},
\eeq
for some singular weights $\vp_i, \psi_i$, where $f = (\om,\eta,\xi)$ and $g_i(h)\asymp |h|^{-1/2}$ is $- \f{1}{2}$ homogeneous
and 
$|| \cdot ||_{C_{g_i}^{1/2} }$ is a seminorm equivalent to the $C^{1/2}$ seminorm.

For more discussion, we refer to 
\cite[Section 2]{ChenHou2023a} and the review 
\cite[Section 3.3.1-3.5]{chen2025singularity}.

\subsection{Estimate of nonlocal terms}
\label{sec:nonlocal}

While the weighted estimate generates sufficiently strong damping terms  $ (\overline M + d_{\vp}(x) \Id) \cdot (W\vp)$ \eqref{eq:damp_sys}, the estimate of nonlocal terms $\uu \vp$ and $(\nabla \uu) \vp$ 
in $\vp \cL_{nloc} \om$ \eqref{eq:damp_sys} \emph{depends} on $\vp$ and cannot be treated perturbatively relative to the local damping terms. 
Among these, the most difficult term is $u_x\bar\theta_x$ in \eqref{eq:etafull}, 
as the remaining $\na \uu$ terms have smaller coefficients, and the $\uu$ terms are more regular and can be treated as compact perturbations to the local operator in \eqref{eq:lin}.  Below, we focus on the top order weighted $C^{1/2}$ estimates.

 Denote by $u_x(x, a, b)$ the localized version of $ u_x$ 
\beq\label{eq:ux_loc_idea}
\bal
 u_{x}(\om) (x, a, b)  & \teq  - \f{1}{\pi} P.V. \int_{  |x_1-y_1| \leq a, | x_2 - y_2| \leq  b }K_1(x - y)  W(y) dy ,  \quad  K_1(s) = \f{s_1 s_2}{|s|^4},
\eal
\eeq
where $W$ is an odd extension of $ \om$ in $y$ from $\R_2^+$ to $\R^2$. To estimate  $u_{x} \psi$ with weight $\psi$, 
\footnote{
In the top order energy estimates for \eqref{eq:EE_top}, we require $C^{1/2}$ estimates 
for $ (\na \uu) \psi$, where $\psi = \psi_1$ is the same weight 
used in the $C^{1/2}$ estimate for $\om$ in \eqref{eq:EE_top}.
This weight $\psi$ differs from $\vp$ in the weighted $L^{\infty}$ estimates in Section \ref{sec:outgo}.
} 
we isolate the most singular part by decomposing  $ u_x =  u_{x, S} +  u_{x, R}$, where $u_{x, S}(x) = u_x (x, a, b)$ captures the singularity at $y=x$ and  $u_{x, R}$ has a non-singular kernel and is more regular.
 For the singular part, we observe that the commutator 
\[
\bal
u_x(\om)(x,a,b) \psi(x) - u_x( \om \psi)(x,a,b)  
= - \f{1}{\pi} \int_{  |x_1-y_1| \leq a,   | x_2 - y_2| \leq  b  } K_1(x-y) W(y)( \psi(x) - \psi(y)) dy 
\eal
\]
is more regular. Therefore, we have the decomposition 
\beq\label{eq:ux_commu1}
u_x(\om)(x, a, b) \psi = u_x( \om \psi)(x, a, b) +
( u_x(\om)(x,a,b) \psi(x) - u_x( \om \psi)(x,a,b)   )
\teq I_1 + I_2.
\eeq

Given that $\om \vp \in L^{\infty}$ with some weight $\vp$, since $K_1(x-y) (\psi(x) - \psi(y)$ has a singularity of order $\f{1}{|x-y|}$, $I_2$ is log-Lipschitz and thus more regular than $I_1$. 
Thus, in the weighted $L^{\infty}$ and $C^{1/2}$ estimates for the energy \eqref{eq:EE_top}, the terms $I_2$ and $u_{x,R}$ are log-Lipschitz, and hence more regular than 
the perturbation $\om, \eta, \xi$, which is $C^{1/2}$ in the energy space. 
For these regular parts, we approximate them by a finite-rank operator $\cK_N$
and further estimate the operator $\cK_N$ using the 
quantitative finite-rank perturbation method in Section \ref{sec:finite_rank}.

\subsubsection{Sharp H\"older estimate via optimal transport  }

For the singular part 
$u_x(\om)(x, a, b)$ in \eqref{eq:ux_loc_idea}
and $I_1$ in \eqref{eq:ux_commu1},  we make the key observation that the H\"older estimate of $\na \uu = \na \na^{\perp}(-\D)^{-1} \om$ with a sharp constant is related to an optimal transport problem. 
In the $C_x^{\al}$ estimate 
\footnote{
\label{foot:Cx}
We refer to $C_x^{\al} $ estimate as estimating $ |f(P) - f(Q)| / |P-Q|^{\al}$ for $P, Q \in \R_+^2$ with $P_2 = Q_2$ but $ P_1 \not= Q_1$. We define $C_y^{\al} $ similarly and denote by $[f]_{C_x^{\al}}, [f]_{C_y^{\al}}$ the seminorms.
}
of $ u_x(x,0) -  u_x(z,0)$ with $\al \in (0, 1)$, using the translation and scaling symmetries, we only need to estimate 
\beq\label{eq:hol_est}
\bal
S & = u_x(\f12,  0) - u_x(-\f12, 0)  = - \f{1}{\pi} P.V. \int_{ \R^2} k(s) W( s_1,  - s_2 ) ds , \\
\eal
\eeq
in terms of $|| W||_{\dot C_x^{\al}}$, where $k(s) = K_1(\f12 -s_1, s_2) - K_1(-\f12-s_1, s_2)$ with $K_1$ from \eqref{eq:ux_loc_idea}
and satisfies $P.V. \int k(s) ds  = 0$.  An estimate of $S$, using $|| W||_{\dot C^{\al}}$, is equivalent to estimating the transportation cost of moving the positive region of $k(y)$ with measure $k^+(y) dy $ to its negative region with measure $k^-(y) dy$ with the \textit{concave} cost function $c(x, y) = |x - y|^{\al} $.
Thus, to obtain a sharp estimate of $S$ \eqref{eq:hol_est}, we seek a measurable transport map 
to obtain a cost as small as possible. We refer to \cite[Section~3.2]{ChenHou2023a} for further discussion of this connection.

For nonlocal operators in 1D with kernel $f$ (e.g., the Hilbert transform), we derive the optimal estimate and the equation of the map:

\begin{lemma}[Lemma 3.6 \cite{ChenHou2023a}]\label{lem:trans}
Suppose that there exists $c \in (a, b)$ such that $f < 0$ on $(a,c)$, $f > 0 $ on $(c, b)$, 
$f |x-c|^{\al} \in L^1_{loc}$ with $\int_a^b f(x) dx = 0$.
For $\al \in (0, 1), g \in C^{\al}(a, b)$, we have
\[
\bal
 \B|\int_a^b f(x) g (x) dx \B| & \leq \int_c^b |f(x)| |x - T(x)|^{\al} dx \cdot || g||_{\dot C^{\al}} ,\\
 \eal
\]
where $T(x)$ solves  $\int_{x}^{T(x)} f(s) ds = 0$.
\end{lemma}

The bound is given by some explicit integral depending on the map. To estimate 
$S$ in \eqref{eq:hol_est}, we analyze the sign of $k(s)$ in \eqref{eq:hol_est} for each fixed $s_2$ and apply Lemma \ref{lem:trans} horizontally to construct the transport map $T(s)$ in $\R^2$, which solves a cubic equation. \footnote{
For fixed $\xx, \zz$, we consider $K_{\D}(y) = K(\xx- \yy) - K(\xx- \zz)$ with kernel $K$ 
associated with the velocity operator $\pa_{ij}(-\D)^{-1}$. 
Using Lemma \ref{lem:trans}, we solve $T(a)$ from $\int_a^{T(a)} K_{\D}(y) d y_i =0 $ for $i=1$ or $i=2$, which reduces to a cubic equation of $T(a)$ when $x_1 = z_1$ or $x_2 = z_2$. Moreover, using translation and scaling symmetries in the H\"older estimate, we can reduce the pairs $(\xx, \zz)$ with $x_i = z_i$ depending on three parameters (e.g. $x_1, z_1, x_2$) to at most one parameter.
}
For any $\xx , \zz\in \R^2_+$ with $x_2= z_2$,  we can obtain
\beq\label{eq:hol_sharp}
  |u_x(\om)(\xx) - u_x(\om)(\zz) | \leq C_* |x_1 - z_1|^{ 1/2} 
  [ \om ]_{C_x^{1/2}} ,
\eeq
where $ [ \cdot ]_{C_x^{1/2}}$ denotes the seminorm (see Footnote \ref{foot:Cx}). The explicit formula of $C_*$ with $C_* \leq 2.55$ and a stronger version of \eqref{eq:hol_sharp} were established in \cite[Lemma 3.1]{ChenHou2023a}.
We develop similar sharp $C_x^{1/2}, C_y^{1/2}$ estimates (see Footnote \ref{foot:Cx})
for other terms in $\na \uu$ using Lemma \ref{lem:trans}. Using the triangle inequality, we obtain the H\"older estimates in $\R^2_+$. We refer the reader to the details of these estimates to \cite[Section 3 \& Appendix B]{ChenHou2023a}.

We generalize the \emph{unweighted} H\"older estimates to localized ones for $u_x(\om)(x, a,b )$ in \eqref{eq:ux_loc_idea},
and further to \emph{weighted} estimates with the \emph{same} constant, \emph{independent} of the weight, using the decomposition \eqref{eq:ux_commu1}, up to more regular terms depending on the weight, 
e.g. $I_2$ in \eqref{eq:ux_commu1}. The sharp H\"older estimates allow us to obtain a larger stability factor, corresponding to $\lam$ in \eqref{eq:intro_EE}, and reduce the number of rank 
in the finite-rank perturbation method in Section \ref{sec:finite_rank} significantly.

\begin{remark}[\bf Constant independent of $\psi$]

It is crucial to commute $\psi$ with the nonlocal operator in \eqref{eq:ux_commu1} so that the constant in the $C^{1/2}$ estimate of $u_x(\om\psi)(x,a,b)$ in terms of $\om\psi$ is \emph{independent} of $\psi$. Using the large damping terms $d_{\vp}(\xx)$ 
(and similarly $d_{\psi}(\xx)$) in \eqref{eq:damp_sys}  generated by the weights in Section \ref{sec:outgo}, we treat $u_x(\om\psi)(x,a,b)$ and related top order terms perturbatively, which in turn yields stability at the top-order level.
The more regular terms are estimated using the finite-rank perturbation method.

\end{remark}

\section{Principle: numerical construction and analytic corrections}\label{sec:principle}

In the remainder of the paper, we focus on the analytic low-rank correction method, as motivated in Section \ref{sec:analytic_low_rank}.
The argument has three conceptually distinct steps.

\vs{0.05in}

\paragraph{\bf Step 1: Numerical construction}
We construct numerical approximations for the profile, the nonlocal velocity,  and  auxiliary
solutions to the linearized equation.  These functions are represented globally
using \emph{explicit} basis functions with sufficient regularity, for example
$C^{4,1}$ B-splines or related piecewise polynomial bases.  We further derive various functions as  linear combinations or products of these functions. The numerical computations determine

\begin{itemize}[leftmargin=2.5em]
\item coefficients of basis functions in various representations, e.g. the profile and  numerical approximations of the stream function;

\item rigorously controlled values or bounds for residual errors and certain derivatives of the represented functions that enter the analytic correction.
\end{itemize}

Once the numerical coefficients are fixed, they define explicit analytic 
functions through the chosen basis representation. 
From that point onward, the following steps are analytic.

\vs{0.05in}

\paragraph{\bf Step 2: Analytic corrections}
The raw numerical objects need not satisfy the exact vanishing conditions required by
the singular weights in the weighted energy estimates. We therefore perform \emph{analytic}, \emph{low-rank}
corrections near $\xx=0$ by subtracting the corresponding localized Taylor expansion polynomial near $\xx=0$,
which involves the following analytic steps 

\begin{itemize}[leftmargin=2.5em]
\item Taylor expansion near the origin,
\item construction of analytic correction modes,
\item decomposition of the perturbation and equations.
\end{itemize}

Because the functions constructed numerically in Step 1 are represented by explicit basis expansions, the needed derivatives of these functions at the origin are defined analytically through the basis functions. The exact vanishing conditions are properties of the corrected function, not of the raw numerical approximation.

\vs{0.05in}

\paragraph{\bf Step 3: Weighted estimates for the corrected functions}
After the analytic corrections, the corrected functions satisfy exact vanishing conditions. 
Hence the weighted estimates near $\xx=0$ reduce to bounding derivatives of the corrected
functions, and these derivatives can be estimated from the basis
representation and the numerical coefficients rigorously.

\section{A model equation for the low-rank correction}\label{sec:model}

In this section, we consider the following model problem. Suppose that the scalar function $W(t, \xx), \xx = (x, y) \in \R \times \R_+$ satisfies the following equation 
\beq\label{eq:lin_model}
\pa_t W = \cL W,
\quad   \cL W \teq - \overline V(\xx) \cdot \na W + \bar c(\xx)  W +  \overline F(\xx)  P(W) ,
\eeq
where $\cL$ is the linear operator, $\overline V= (\overline V_1, \overline V_2)$ is a vector-valued function, the function $P(W)$ depends on $W$ via a rank-one condition
\[
   P(W) = \int_{\R^2_{++}} W( \yy ) \bar p(\yy) d \yy,
\]
for some function $\bar p(\yy)$, $\bar c$ is a scalar coefficient, and $\overline F$ is the error term. We assume that $W, \overline V_1, \overline F$ are odd in $x$, $\overline V_2$ is even in $x$, and $\overline F, \overline V$ satisfy
\beq\label{eq:van_F}
  \overline F( \xx) = O( |\xx|^2), \quad \pa_{xy} \overline F(0) \neq 0,
  \quad \overline V_1( \xx) = c_1 x + O(|\xx|^2),
  \quad \overline V_2( \xx) = c_2 y + O(|\xx|^2),
\eeq
near $ \xx = 0$. The transport term $\overline V \cdot \na  W $ in \eqref{eq:lin_model} models the transport term $(\bar c_l \xx + \bar \uu) \cdot \na   $ in \eqref{eq:lin}.

For the model problem \eqref{eq:lin_model}, we assume that the functions $\overline V(\xx), \bar c(\xx), \overline F(\xx)$ are represented as linear combinations of various \emph{explicit} and smooth enough basis functions $f_i(\xx)$ :
\bseq\label{eq:ass_f_bar}
\beq\label{eq:ass_f_bar:a}
 \bar f(\xx) =  \sum\nolimits_{i \leq n_f}  c_i f_i(\xx) , \quad \mbox{for \ } \bar f = \overline V, \bar c , \overline F ,
\eeq
with $n_f < \infty$, or the linear combinations or products of the basis functions 
\footnote{
One can expand the form \eqref{eq:ass_f_bar:b} to obtain the expression \eqref{eq:ass_f_bar:a}.
}
\beq\label{eq:ass_f_bar:b}
   \bar f = \sum_{n \leq N } (  \sum_{i \leq n_1} c_{n, i} f_{n, i}(\xx) ) \cdot (  \sum_{i \leq n_2} d_{n, i} g_{n, i}(\xx) ) ,
\eeq
with $C^{4,1}$ basis functions $f_{n, i}(\xx), g_{n, i}(\xx)$.
\eseq

In \cite{ChenHou2023a}, examples of basis functions $f_i(\xx)$ include 6-th order B-spline functions, which are piecewise polynomials with compact support, and smooth cutoff functions. Due to the regularity of the basis functions, for \emph{any} coefficients $c_{n,i}, d_{n,i}$, we have $\bar f \in C^{4, 1}$.  Given the basis functions $f_i(\xx)$, the numerical construction of profiles or approximate solutions determines the coefficients $c_i$, which are real numbers in $\R$. 
We emphasize that these functions $\overline V, \bar c, \overline F$ are defined everywhere, not just 
\emph{on the grid points} that are used for numeric discretization.

We consider initial perturbation with cubic vanishing order near $\xx=0$
\beq\label{eq:init_van}
  W_0(\xx) = O(|\xx|^3) .
\eeq
Due to the nonlocal part $P(W) \overline F$ and \eqref{eq:van_F}, equation \eqref{eq:lin_model} only preserves quadratic vanishing order $W(t, \xx) = O(|\xx|^2)$ near $\xx=0$.

\vs{0.1in}
\paragraph{\bf Decomposition}

 To improve the vanishing order of $W$ near $\xx=0$, 
 we track the Taylor expansion of various functions near $\xx=0$.
 Let $\chi(\xx)$ be a $C^{4,1}$ function, odd in $\xx$, with compact support and satisfying
 \beq\label{def:chi}
   \chi(\xx) = x y + O( |\xx|^3) .
 \eeq

We decompose $W = W_1 + W_2$ with $W_i$ solving 
 \bseq\label{eq:lin_W}
 \begin{align}
    \pa_t W_1 + \overline V(\xx) \cdot \na W_1 & = \bar c(\xx) W_1 + ( \overline F - \overline F_{xy}(0) \chi ) P(W_1), 
    \label{eq:lin_W1} \\
  \pa_t W_2 + \overline V(\xx) \cdot \na W_2 & = \bar c(\xx) W_2 +  \overline F P(W_2) + \overline F_{xy}(0) \chi \cdot P(W_1),   \label{eq:lin_W2}   
\end{align}
 \eseq
 with initial data 
 \[
         W_1 |_{t = 0} = W_0, \quad W_2 |_{t = 0} = 0.
 \]

Since $\overline F$ is odd in $x$, we obtain $\pa_x^2 \overline F(0) = \pa_y^2 \overline F(0) = 0$. 
Moreover, since $\overline F$ has the analytic form \eqref{eq:ass_f_bar}, we can define 
the derivatives of  $ \pa^{\al} \overline F(0)$ analytically by applying the derivatives to the basis functions $f_i$, e.g. $\pa^{\al} f_i(\xx)$. Therefore, by a Taylor expansion, we obtain the vanishing order
\beq\label{eq:F_cor1}
 \overline F -  \pa_{xy} \overline F(0) \chi = O(|\xx|^3),
\eeq
near $\xx=0$ analytically. Note that we \emph{do not} need to compute the numerical value of $\pa_{xy} \overline F(0)$ to obtain the above vanishing order. The derivative is defined analytically through the basis representation, e.g.
\beq\label{eq:F_cor1_analy}
 \pa_{xy} \B( \sum\nolimits_{i \leq n_f} c_i f_i(\xx)  \B) \B|_{\xx=0}
 = \sum\nolimits_{i \leq n_f} c_i \pa_{xy} f_i(0).
\eeq
Therefore, the vanishing condition \eqref{eq:F_cor1} follows from the representation \eqref{eq:ass_f_bar} and the regularity of the basis, which is an exact analytic statement.

This is the first important lesson of the toy model: the correction is not
performed on \emph{raw numerical values}, but on a \emph{function}
$\overline F$ represented explicitly in a basis.

As a result, the vanishing order $W_1(t, \xx) = O(|\xx|^3)$ is preserved by \eqref{eq:lin_W1}. In other words, the bad quadratic mode in $\overline F$ has been removed analytically in the equation for $W_1$. On the other hand, equation \eqref{eq:lin_W2} still carries the delicate nonlocal forcing and therefore preserves only quadratic vanishing order $W_2(t, \xx) = O(|\xx|^2)$.

\subsection{Duhamel formula and computer-assisted construction}

We perform weighted energy estimates for $W_1$ in \eqref{eq:lin_W1}. For $W_2$, using the Duhamel formula, we obtain
\beq\label{eq:W2_form}
  W_2(t) = \int_0^t e^{\cL (t- s)}  \overline F_{xy}(0) \chi P(W_1(s)) d s 
  = \int_0^t  \overline F_{xy}(0)  P(W_1(s)) \cdot e^{\cL (t- s)} \chi  d s ,
\eeq
where $\cL$ is the linear operator defined in \eqref{eq:lin_model}. 
Note that we do not need to construct and estimate the \emph{full} semigroup $e^{\cL s}$ generated by $\cL$, but only $e^{\cL s}\chi_\dag$. Since $\cL$ is complicated, 
we cannot estimate $e^{\cL(t-s)} \chi$ directly. 
Instead, we follow the principle in Section \ref{sec:principle} by first constructing a numerical approximation, and then performing analytic corrections.

We construct a numerical 
approximation $\wh g(t, \xx)$ for $e^{\cL (t- s)} \chi$ by solving the linear equation numerically 
\beq\label{eq:lin_g}
  \pa_t g = \cL g , \quad g(0, \xx) = \chi. 
\eeq
We use $\widehat {\cdot}$ to indicate that a solution is constructed numerically. 
The solution $\wh g(t, \xx)$ has a representation similar to \eqref{eq:ass_f_bar}.  
 It is defined everywhere for $t$ and $\xx$, $g(t, \cdot)$ has spatial $C^{4, 1}$-regularity, is piecewise smooth in $t$, and is Lipschitz in $t$ globally. We discuss more details of the construction of $\wh g$ in Section \ref{sec:numer}. 

With the approximate solution $\wh g$, we modify the formula for $W_2$ as 
\beq\label{eq:W2_form_hat}
 \wh W_2 = \int_0^t  \overline F_{xy}(0)  P(W_1(s)) \cdot \wh g(t- s, \xx)  d s.
\eeq
Since $\wh g(t, \xx)$ is defined globally and is Lipschitz in $t$, we can define the above integral \emph{analytically}.
We remark that the above integration \emph{is not a numerical integration}, which typically depends on 
quadrature rules.

Due to numerical error, $\wh g$ does not solve \eqref{eq:lin_g}  \emph{exactly}, and $\wh W_2$ does not solve \eqref{eq:lin_W2} \emph{exactly}. To quantify the error rigorously,  we introduce the following residual operator associated with the solution $\wh g$ and $W_1$
\beq\label{eq:cR}
\cR( W_1) \teq P(W_1(t)) \overline F_{xy}(0) \cdot ( \wh g(0, \xx) - \chi) + \int_0^t \overline F_{xy}(0) P(W_1(s)) \cdot  (\pa_t -\cL) \wh g( t- s) d s.
\eeq
Since $\wh g(t- s, \xx)$ is defined everywhere via the globally defined basis functions and is sufficiently regular, we can define 
$(\pa_t -\cL) \wh g( t- s) $ and the above integral analytically.

Using the analytic expressions \eqref{eq:W2_form_hat} and \eqref{eq:cR}, we obtain the following equation for $\wh W_2$, defined in \eqref{eq:W2_form_hat}, \emph{exactly}
\bseq\label{eq:lin_Wm}
\beq
  \pa_t \wh W_2 = \cL \wh W_2 + \overline F_{xy}(0) \chi \cdot P(W_1) + \cR(W_1).
\eeq

To ensure that $W_1 + \wh W_2$ solves the original model \eqref{eq:lin_model}, we modify the decomposition in \eqref{eq:lin_W} by subtracting the residual operator $\cR(W_1)$ from \eqref{eq:lin_W1}
\beq
    \pa_t W_1 + \overline V(\xx) \cdot \na W_1  = \bar c(\xx) W_1 + ( \overline F - \overline F_{xy}(0) \chi ) P(W_1) 
    - \cR(W_1).
\eeq
\eseq

We remark that, given any function $\wh g$ with suitable regularity, we perform the above 
 reformulation \eqref{eq:lin_Wm} of \eqref{eq:lin_W} analytically, 
 and there is no additional numerical error 
 other than those in $\bar \cR$ 
 \footnote{
The numerical error in constructing $\wh g$ enters \eqref{eq:lin_Wm} via  $ \wh g(0, \xx) - \chi, 
(\pa_t - \cL) \wh g(t- s, \xx)$ and the operator $\cR(W_1)$ \eqref{eq:cR}.
}
 in the passage from the original decomposition \eqref{eq:lin_W} to the modified one \eqref{eq:lin_Wm}. In particular, once the approximate function $\wh g$ is fixed, the reformulation \eqref{eq:lin_Wm} is exact: no additional approximation is introduced beyond the residual terms already contained in $\wh g(0, \xx)-\chi$ and $(\pa_t-\cL)\wh g$.

\paragraph{\bf Constraints for $\wh g$}

To ensure that $W_1$ vanishes cubically near $\xx=0$ so that it belongs to the energy class, we need to construct $\wh g(t, \xx)$ with
\beq\label{eq:g_cond}
\wh g(t, \xx) = O(|\xx|^2), \quad  \wh g(0, \xx) - \chi = O(|\xx|^3), \quad   (\pa_s -\cL) \wh g( s, \xx) = O(|\xx|^3). 
\eeq

We remark that the \emph{exact} solution to \eqref{eq:lin_g} satisfies the above properties.  To achieve the above properties, we follow the principle in Section \ref{sec:principle} by first constructing a numerical approximation for $g$ to \eqref{eq:lin_g} and then perform analytic corrections.

\subsection{Approximate space-time solution and corrections}\label{sec:numer}

In this section, we explain the idea and methods in \cite[Section 4.2, 4.3]{ChenHou2023a}
and \cite[Section 3]{ChenHou2023b} to construct $\wh g$ satisfying \eqref{eq:g_cond}.
The key idea is to perform low-rank \emph{analytic} correction near $0$ using suitable Taylor expansion. The construction of $\wh g$ and the corrections is discussed in detail in \cite[Section 3]{ChenHou2023b}.

\vs{0.1in}
\paragraph{\bf Step 1. Numerical construction}

We construct the approximate solution to \eqref{eq:lin_g} numerically at discrete time $t_k = k h$ using 
representation \eqref{eq:ass_f_bar:a} with piecewise polynomial basis functions $B_{1,i}(x), B_{2,j}(y)$:
\beq\label{eq:g_0form}
   \wh g^{(0)}(t_k, x, y) = \sum\nolimits_{ij} a_{k, ij} B_{1, i}( x) B_{2, j}(y)
\eeq
and $t_k \leq T$, where $T$ is the final time up to which we construct the numerical solution. 
The subscripts $1$ and $2$ in $B_{1,i}$ and $B_{2,j}$ indicate the $x$- and $y$-directions, respectively. 
By interpolating $a_{k, ij}$ using piecewise cubic polynomials (see \cite[Section 3.4]{ChenHou2023b}), 
we extend $\wh g^{(0)}(t_k, x, y)$ to a globally defined function $ \wh g^{(0)}(t, x, y)$ with 
compact support in time $t\in [0, T]$. Note that in each time interval $[t_k, t_{k+1}]$, 
$ \wh g^{(0)}(t, x, y)$ satisfies 
\beq\label{eq:g_1form}
\wh g^{(0)}(t, x, y) = \sum\nolimits_{ij} a_{ ij}(t) B_{1, i}(x) B_{2, j}(y), 
\eeq
where $ a_{ij}(t)$ is a cubic polynomial in $t$ with coefficients being a \emph{linear combination} of $a_{k, ij}$ in \eqref{eq:g_0form}. 

We require that $B_{1, i}(x)$ is odd in $x$ so that $ \wh g^{(0)}$ is odd in $x$ 
and vanishes $O(|\xx|)$ near $\xx=0$. 
\footnote{
To construct the approximation for $\xi$ to the linearized Euler equations, we choose basis $B_{1,i}(x)$ for $\xi$ even in $x$.  
}
We choose the basis functions $B_{1,i}(x), B_{2,j}(y) \in C^{4, 1}$ so that $  \wh g^{(0)} \in C^{4, 1}$.

\vs{0.1in}
\paragraph{\bf Step 2. First analytic correction}

While the \emph{exact} solution to \eqref{eq:lin_g} vanishes $O(|\xx|^2)$ dynamically, due to numerical error, the solution $ \wh g^{(0)}$ only vanishes like $O(|\xx|)$ near the origin and does not preserve the $O(|\xx|^2)$ vanishing order. The first analytic correction is used to restore the
quadratic vanishing order that the exact solution should satisfy.

 Let $\chi_1$ be a $C^{4,1}$ function that is odd in $x$, has a compact support, and satisfies 
\[
 \chi_1(\xx) = x + O(|\xx|^3) ,
\]
near $x=0$. We perform the first \emph{analytic} correction on $ \wh g^{(0)}$:
\beq\label{eq:g_cor1}
 \wh g^{(1)}(t, \xx ) \teq  \wh g^{(0)}(t, \xx ) - \pa_x  \wh g^{(0)}(t, 0) \cdot \chi_1( \xx)  , 
\eeq
and obtain $ \wh g^{(1)} \in C^{4, 1}$. We obtain the vanishing order $ \wh g^{(1)}(t, \xx ) = O(|\xx|^2)$
for the same reason as \eqref{eq:F_cor1} and the analytic structure \eqref{eq:g_0form}.  The value $\pa_x  \wh g^{(0)}(t, 0)$ is of the order of the numerical error and is expected to be very small. 
We achieve the first condition in \eqref{eq:g_cond}.

This is exactly the place where the distinction between \emph{analytic
expression} and \emph{numerical evaluation} matters. The correction uses the exact analytic derivative of the explicitly represented function $\wh g^{(0)}$. For example, if
\[
\wh g^{(0)}(t,x,y)=\sum\nolimits_{ij} a_{ij}(t) B_{1,i}(x) B_{2,j}(y),
\]
then
\[
\pa_x \wh g^{(0)}(t,0,0)=\sum\nolimits_{ij} a_{ij}(t) (\pa_x B_{1,i})(0) B_{2,j}(0),
\]
which is an exact analytic quantity determined by the basis functions and the coefficients. What is not needed is a separate numerical evaluation of that derivative.

\vs{0.1in}
\paragraph{\bf Step 3. Second analytic correction}

While the \emph{exact} solution to \eqref{eq:lin_g} satisfies 
$g(0, \xx) - \chi \equiv 0, (\pa_s - \cL) g(s, \xx) \equiv 0$, and thus the cubic vanishing condition 
near $0$ as in \eqref{eq:g_cond}, due to numerical error, $ \wh g^{(1)}$ constructed in the first two steps does not satisfy \eqref{eq:g_cond}. The second analytic correction is used to restore the
cubic vanishing of the error that the exact solution should satisfy. 

We introduce the error parts:
\beq\label{def:err}
  \cE_0( \xx) \teq \wh g^{(1)}(0, \xx) - \chi ,  \quad  \cE(s, \xx) \teq  (\pa_s -\cL)  \wh g^{(1)}(s, \xx) .
\eeq
Since $\wh g^{(1)}(t, \xx)$ has spatial $C^{4, 1}$ regularity by \eqref{eq:g_1form} and 
\eqref{eq:g_cor1}, we obtain the spatial regularity $\cE(s , \cdot) \in C^{3, 1}$ for any $s \geq 0$ 
and $ |\pa^{\b} \cE(s)|$ is bounded in $s$ 
for any $|\beta| \leq 2$.

By definition of $\cL$ in \eqref{eq:lin_model},
the conditions in \eqref{eq:van_F}, and $ \wh g^{(1)}(t, \xx) = O(|\xx|^2)$ near $\xx=0$, 
using a Taylor expansion, we obtain the following leading order terms near $\xx=0$
\beq\label{eq:err_2nd}
  \cE_0( \xx) = \pa_{xy} \cE_0(0) x y + O(|\xx|^3), \quad  \cE(s, \xx)
  = \pa_{xy}\cE(s, 0) x y + O(|\xx|^3) . 
\eeq
The scalar $ \pa_{xy} \cE_0(0)$ and the function $\pa_{xy}\cE(s, 0)$
 depend on $\wh g^{(1)}, \overline F, \overline V, \bar c(\xx)$ and have a size of numerical error. 

We introduce a function $\chi_2 \in C^{4,1}$, that is odd in $x$, has a compact support, and satisfies 
\footnote{ 
In \cite[Section 4.2, Section 4.3]{ChenHou2023a} and more details in \cite[Section 3]{ChenHou2023b},
the nonlocal correction term $P(f)$ near $\xx=0$ is given by 
$P(f) = - \pa_{xy}(-\D)^{-1} f(0)$. We choose $\chi_2(\xx) = -\D( \f{1}{6} x y^3 \kp(x) \kp(y))$ with $\kp(x)$ being smooth cutoff function supported near $x=0$ and $\kp(x)=1$ near $x=0$ to obtain 
$\pa_{xy}(-\D)^{-1} \chi_2(0) = 0$ and $\chi_2(\xx) = x y + O(|\xx|^3)$ near $\xx=0$.
}
\beq\label{def:chi2}
\chi_2(\xx) = x y + O(|\xx|^3), \quad  P(\chi_2) = 0.
\eeq

To further eliminate the quadratic error, we perform the second \emph{analytic} correction
\beq\label{def:g2}
  \wh g^{(2)}(t, \xx) = \wh g^{(1)}(t, \xx) + a(t) \chi_2.
\eeq
From \eqref{eq:g_0form}, \eqref{eq:g_cor1}, and the above formula, we obtain that 
$\wh g^{(2)}(t, \xx)$ has the representation \eqref{eq:ass_f_bar:a} with $C^{4, 1}$ basis functions. 

\paragraph{\bf Constructing and solving $a(t)$ analytically}

We choose $a(t)$ to eliminate the quadratic parts near $\xx=0$ in \eqref{eq:err_2nd} by imposing the following \emph{analytic} condition on $a(t)$:
\beq\label{eq:van_g2}
  \pa_{xy} (   \wh g^{(2)} - \chi )|_{s=0, \xx=0 } = 0,
  \quad \big( \pa_{xy}  (\pa_s - \cL) \wh g^{(2)}\big)(s, 0) = 0.
\eeq

Due to symmetry, using \eqref{eq:err_2nd} and \eqref{def:g2}, we rewrite the above analytic conditions equivalently as 
\beq\label{eq:ODE1}
 a(0) + \pa_{xy} \cE_0(0) = 0,
 \quad  \pa_{xy} ( \pa_s - \cL) a(s) \chi_2 |_{s=0, \xx=0} + \pa_{xy} \cE(s, 0) = 0.
\eeq

For any function $f \in C^3$, odd in $x$, and $f = O(|\xx|^2)$, from \eqref{eq:lin_model}, we derive 
\[
   \pa_{xy} ( \cL  f )(0)=  (\bar c(0) - \pa_x \overline V_1(0) - \pa_y \overline V_2(0) ) \pa_{1 2} f(0)
 +  \pa_{1 2}\overline F(0) \cdot P( f ).
\]

Using the above formula,  $\pa_{xy} \chi_2(0) = 1, P(\chi_2) = 0$ by \eqref{def:chi2}, 
we rewrite equations \eqref{eq:ODE1} as
\beq\label{eq:ODE2}
\bga
 0  = \f{d}{d s} a(s) - a(s) \pa_{xy} \cL \chi_2(0)
 =  \f{d}{d s} a(s) - a(s) \big( \bar c(0 ) - \pa_x \overline V_1(0) - \pa_y \overline V_2(0)   \big)
 + \pa_{xy} \cE(s, 0), \\
 a(0)  = - \pa_{xy} \cE_0(0).
 \ega
\eeq

The above relation gives a \emph{linear} ODE in $a(s)$, which we solve \emph{analytically}
\beq\label{def:a}
  a(t) = - e^{\bar \lam t}  \pa_{xy} \cE_0(0)
- \int_0^t e^{ \bar \lam (t- s) }  \pa_{xy} \cE(s, 0 )  d s,
\quad \bar \lam = \bar c(0 ) - \pa_x \overline V_1(0) - \pa_y \overline V_2(0).
\eeq

Since $\cE(s,{\bf x}) $ has $C^{3,1}$ spatial regularity in ${\bf x}$, and $\na^{\beta} \cE(s, {\bf x})$ is bounded in $s$ for $|\beta| \leq 2$, we can define the above integral \emph{analytically}. The formula \eqref{def:a} is one of the key points in the entire corrections. It
shows that the second correction step is not another numerical approximation.
Rather:

\begin{itemize}[leftmargin=2.5em]
\item the numerics identify the defect coefficients
      $\partial_{xy}E_0(0)$ and $\partial_{xy}E(s,0)$ via \eqref{eq:err_2nd};
\item the function $a(t)$ is then obtained by an \emph{exact analytic formula} \eqref{eq:ODE2};
\item once $a(t)$ is fixed, the corrected function $\wh g^{(2)}$ is 
defined analytically for $t \in \R_+, \xx \in \R_2^+$;

\item the cubic vanishing of $\wh g^{(2)}$ follows from its Taylor expansion.
\end{itemize}

\paragraph{\bf Summary }

After we construct the raw numerical approximation $\wh g^{(0)}$ in step 1, the subsequent correction step is entirely analytic: one uses the smallness of this error to solve a \emph{perturbative} problem and eliminate the residual defect, thereby recovering the exact vanishing conditions.

We note that while $\wh g^{(1)}(t), \cE(t)$ have a compact support in $t$, due to the above 
convolution, $a(t)$ \emph{does not} have a compact support.
From the above definition, $a(t)$ depends on $\wh g^{(1)}$ via the error term $\cE_0, \cE(s)$ defined in \eqref{eq:err_2nd}, and it has a size of numerical error. 
We refer to Example 2 in Section \ref{sec:how} for the weighted estimate of the corrected error.

\begin{remark}[\bf Regularity of the basis functions for the model problem]

To perform the correction for the model problem, we only require $C^{3, 1}$ regularity for the basis functions, such as those in \eqref{eq:ass_f_bar}, in representing the coefficients $\overline V(\xx), \bar c(\xx), \overline F(\xx)$ in \eqref{eq:lin_W},
 and approximate solution $\wh g^{(0)}$ in \eqref{eq:g_0form}. 
 The resulting error, e.g. $\cE(s, \xx)$ in \eqref{def:err}, has $C^{2, 1}$ regularity and is sufficient for later estimates.
 In the above presentation, we use $C^{4,1}$-regularity for the basis function so that it is 
 consistent with the setting in \cite{ChenHou2023a,ChenHou2023b}.

\end{remark}

\begin{remark}[\bf Higher vanishing order by finite rank correction]

We can generalize the decomposition in \eqref{eq:lin_Wm} to obtain 
$W_1$ preserving higher vanishing order $O(|\xx|^k)$ with large $k$ near the origin and $\wh W_2$ constructed with computer assistance. The key idea is to perform similar higher order Taylor expansion and analytic, finite-rank correction near $\xx=0$.

\end{remark}

\subsection{How do numerical values enter the estimates}\label{sec:how}

In this subsection, we briefly explain how numerical information enters the estimates once the exact local cancellation has already been enforced analytically. We refer to \cite[Section 5.7]{ChenHou2023a} and \cite[Section 3.5]{ChenHou2023b} for more details.

In the weighted energy estimates \cite[Section 5]{ChenHou2023a}, we choose a singular weight $\vp$ with order $|\xx|^{-3}$ near $\xx=0$. For a function $f$ with vanishing order $O(|\xx|^3)$ near $\xx=0$,
for $|\xx|\leq 1$, we obtain
\footnote{
We derive all constants in the weighted estimate \emph{explicitly} in \cite{ChenHou2023b}, e.g. 
\cite[Appendix E]{ChenHou2023b}.
}
\beq\label{eq:key_weight_est}
 | |\xx|^{-3} f | \les 
  \| \na^3 f \|_{L^{\infty}( |\xx|\leq 1) } .
\eeq

Thus, to bound the weighted norm of various terms, e.g. $ \wh g(0, \xx ) - \chi , 
(\pa_s - \cL) \wh g(s, \xx)$, we only need to obtain the $C^3$ bounds of these terms. Since we use 
explicit basis functions for various terms, e.g. \eqref{eq:g_0form}, we apply the derivatives to the smooth basis functions, e.g.
\[
   \na^3 \wh g^{(0)}(t_k, x, y) = \sum\nolimits_{ij} a_{k, ij} \na^3 ( B_{1,i}(x) B_{2,j}(y) ).
\]
By choosing basis functions $B_{1,i}(x)$ and $B_{2,j}(y)$ with sufficient regularity, for example of class $C^{4,1}$, and by tracking rigorous bounds for the coefficients $a_{k,ij}$, we obtain piecewise $C^3$ bounds for the relevant functions. In particular, although the computation of $a_{k,ij}$ involves numerical and round-off errors, these errors are \emph{not} directly amplified by the singular weight, since the local corrections are performed analytically once the coefficients have been computed
and we use estimates similar to \eqref{eq:key_weight_est}. In \cite[Appendices C and D]{ChenHou2023b}, we use numerical analysis to rigorously establish piecewise derivatives bounds 
for polynomials and explicit basis functions, which are used for representing 
different functions constructed in Step 1 in Section \ref{sec:principle}. 

\vs{0.1in}
\paragraph{\bf Rigorous piecewise bounds}
We present one such estimate to illustrate the methods and ideas for rigorous piecewise bounds.
To obtain a tight piecewise bound of a function $f$ in $D_{ij} \equiv [x_i, x_{i+1}] \times [y_j, y_{j+1}] $, we have
\beq\label{eq:err_2nd_est}
\bal
  \max_{x, y \in  D_{ij}}   |f(x, y)|   \leq  \max_{ i, j = 1,2}  | f(x_i, y_j |  + \tfrac18 h_1^2   ||\pa_x^2 f||_{ L^{\infty}(D_{ij})}
+ \tfrac18  h_2^2   ||\pa_y^2 f||_{ L^{\infty}(D_{ij})} ,
\eal
\eeq
where  $ h_1 = x_{i+1} - x_i$ and $ h_2 = y_{j+1} - y_j$.
If $f$ is a  polynomial, e.g. Bspline basis functions, the high order derivatives vanish. 
Since we can evaluate the grid point values $f(x_i, y_j)$ using the representation formula
\eqref{eq:ass_f_bar}, we can recursively use \eqref{eq:err_2nd_est} to estimate the derivatives from high order to low order. For non-polynomial basis functions, we first obtain rough bounds on higher-order derivatives, and then refine the piecewise estimates using \eqref{eq:err_2nd_est} by further partitioning the domain with a finer mesh $h_i$.

Below, we use two examples to discuss the rigorous weighted estimates.

\vs{0.1in}

\paragraph{\bf Example 1}

We use a simple example to illustrate the estimates. Suppose that 
\[
  \overline F = \bar f \cdot \bar g ,
  \quad \bar f = c_1 x + \sum\nolimits_{i + j = 2} c_{ij} x^i y^j,
  \quad \bar g = d_1 y + \sum\nolimits_{i + j = 2} d_{ij} x^i y^j ,
\]
 the coefficients $c_{ij}, d_{ij}$ are given as floating point values, and $c_1, d_1 \neq 0$. 
Clearly, we have $\pa_{xy} \overline F(0) = c_1 d_1$. We aim to bound $(\overline F - \pa_{xy} \overline F(0) \chi ) |\xx|^{-3}$ for $|\xx| < 1$, where $\chi$ is a function satisfying \eqref{def:chi}.

Due to numerical error, the numerical evaluation of $\pa_{xy}\overline F(0)$ may differ from its analytic value $\pa_{xy}\overline F(0)=c_1d_1$. 
However, using the \emph{analytic} definition of $\pa_{xy} \overline F(0)$, we obtain the analytic vanishing conditions  $\overline F -\pa_{xy} \overline F(0) \chi = O(|\xx|^3)$. Thus, using the Leibniz rule, we obtain
\[
  \max_{|\xx| < 1}  |\overline F -\pa_{xy} \overline F(0) \chi | \cdot |\xx|^{-3}
  \les \max_{|\xx| < 1} ( | \na^3  \overline F |  + |\pa_{xy} \overline F(0)| \cdot |\na^3 \chi| )
  \les \max_{|\xx| < 1} ( \sum_{k + l = 3} |\na^k \bar f| \cdot |\na^l \bar g| 
    + |\pa_{xy} \overline F(0)| \cdot |\na^3 \chi| ).
\]

The upper bound is \emph{finite for any} value $|\pa_{xy} \overline F(0)|$. 
The numerical evaluation enters the rigorous estimates by bounding the upper bounds. Using error estimates from numerical analysis, we obtain sharp piecewise bounds for $\na^i \bar f, 
\na^j \bar g, \na^k \chi$.
Using these bounds and the above estimate, we bound $ |\overline F -\pa_{xy} \overline F(0) \chi | \cdot |\xx|^{-3}$ rigorously by some bounded constant.

Away from $\xx=0$, we derive piecewise bounds for $\overline F - \pa_{xy} \overline F(0) \chi$ and then derive the weighted estimates. In this regime, although we still cannot numerically evaluate $ \pa_{xy} \overline F(0)$ exactly, the error is not amplified by the weight significantly because the weight is less singular away from $\xx=0$. By optimizing these two estimates, we obtain sharp piecewise bounds for 
$|\overline F -\pa_{xy} \overline F(0) \chi | \cdot |\xx|^{-3}$. 

An important lesson of the above example is that although numerical evaluation may not produce $\pa_{xy} \overline F(0)$ \emph{exactly}, this does \emph{not} lead to an uncontrolled weighted estimate. The crucial point is that the cancellation uses the exact analytic derivative of the represented function $\overline F$, rather than a raw floating-point approximation. Once the exact vanishing order is enforced analytically, the weighted estimate becomes rigorous and bounded.

\vspace{0.1in}

\paragraph{\bf Example 2: weighted estimates for the error}

In this example, we discuss the weighted estimates for the corrected error $ (\pa_s - \cL) \wh g^{(2)}$ for $\wh g^{(2)}$ defined in \eqref{def:g2}. 
Since $\wh g^{(2)} = \wh g^{(1)} + a(t) \chi_2$ \eqref{def:g2}, using the cubic vanishing order 
\eqref{eq:g_cond} for $\wh g = \wh g^{(2)}$, \eqref{eq:key_weight_est}, 
and triangle inequality, we obtain
\beq\label{eq:tri_split}
\max_{|\xx| \leq 1} |\xx|^{-3} (\pa_s - \cL) \wh g^{(2)} |
\les \max_{|\xx|\leq 1} | \na^3 (\pa_s -\cL) \wh g^{(2)} |
\les \max_{|\xx|\leq 1} ( | \na^3 (\pa_s -\cL) \wh g^{(1)} |
+ |a(t)| \cdot |\na^3 (\pa_s -\cL)  \chi_2  | ).
\eeq

Since $\wh g^{(1)}, \chi_2$ are represented using explicit basis functions, we can bound 
$ | \na^3 (\pa_s -\cL) \wh g^{(1)} | , |\na^3 (\pa_s -\cL)  \chi_2  |$ rigorously. 
For $a(t)$, we bound it using the \emph{analytic} formula \eqref{def:a} to obtain
\[
  |a(t)| \leq  e^{\bar \lam t}  |\pa_{xy} \cE_0(0)|
+ \int_0^t e^{ \bar \lam (t- s) }  |\pa_{xy} \cE(s, 0 )|  d s ,
\quad \bar \lam = \bar c(0 ) - \pa_x \overline V_1(0) - \pa_y \overline V_2(0).
\] 

For the ODE system in \cite[Section 4]{ChenHou2023a} and \cite[Section 3.3]{ChenHou2023b} near $\xx=0$, which corresponds to \eqref{eq:ODE2}, we have $\bar \lam < -5$. 
Thus, by bounding $|\pa_{xy} \cE_0(0)|$ and $ |\pa_{xy} \cE(s, 0 ) |$ for $s$ in the support of 
$\cE(s, 0)$ 
\footnote{
In \cite{ChenHou2023a,ChenHou2023b}, the approximate solution $\wh g^{(0)}, \wh g^{(1)}$ and error 
$\cE(s, 0)$ has temporal support $s \in [0,12]$.
}
and 
 using the exponential decay $e^{\bar \lam s}$ in $s$, we bound the factor $a(t)$ for \emph{any} $t \geq 0$.

While the triangle inequality used in \eqref{eq:tri_split} leads to an overestimate
of $ \na^3 (\pa_s -\cL) \wh g^{(2)}$, since $a(t)$ has a size of the error, 
such an overestimate is minor. Away from $\xx=0$, similarly, we estimate 
\[
| |\xx|^{-3} (\pa_s - \cL) \wh g^{(2)} |
\leq | |\xx|^{-3} (\pa_s - \cL) \wh g^{(1)} |
+ |a(t) | \cdot | |\xx|^{-3} (\pa_s - \cL) \chi_2 | .
\]

By deriving sharp piecewise bounds for $(\pa_s - \cL) \wh g^{(1)}, (\pa_s - \cL) \chi_2 $, we obtain 
another estimate for the weighted error. By optimizing the above two estimates, we obtain a sharp piecewise bound for the weighted error rigorously. 

An important lesson in the above example is that we \emph{do not} need to evaluate 
the coefficient for the correction $a(t)$. Since $a(t)$ is very small, using triangle inequality and the above estimates, we only need to obtain an \emph{upper bound} for $a(t)$. 

For the Boussinesq equations, the related estimates are developed in details in 
\cite[Section 3.5]{ChenHou2023b}.

\vspace{0.1in}
\paragraph{\bf Summary}
We perform \emph{analytic} derivation to ensure suitable vanishing order and enforce \emph{exact} conditions, such as \eqref{eq:g_cond}; we use high order regularity of the numerical solution 
and the \emph{exact} vanishing conditions to control the weighted norm of various terms rigorously. 
Using the above ideas and methods, we achieve robust weighted energy estimates 
around numerically constructed solutions, e.g. $\bar \om, \bar \th$ in \eqref{eq:lin}, and the approximate solution $\wh g$.

\subsection{Generalization}

For the system \eqref{eq:lin}, in  \cite[Sections 4.2, 4.3]{ChenHou2023a},  we perform similar corrections and decompositions for the perturbation near $\xx=0$: $W = (\om, \eta, \xi) = (\om_1, \eta_1, \xi_1) + (\om_2 , \eta_2, \xi_2)$, to obtain cubic vanishing order for $W_1 = (\om_1, \eta_1, \xi_1) $. For example, we approximate $\olin \cF_1(\xx) = \overline F_{\om}(\xx) $ by $ \pa_{xy} \olin \cF_1(0) \chi$ with $\chi = x y + O(|\xx|^3)$ near $\xx =0$; we choose $c_{\om}(f) = \pa_x u(f)(0) = - \pa_{xy}(-\D)^{-1} f$ and decompose the linear nonlocal terms involving $\uu = (u, v), c_{\om}$ in \eqref{eq:linomega} as
\[
  c_{\om}(\om) \bar \om - \uu \cdot \na \bar \om = - ( u - \pa_x u(0) x, v + \pa_x u (0) y ) \cdot \na \bar \om
  + c_{\om}(\om) (\bar \om - x \pa_x \bar \om + y \pa_y \bar \om),
\]
where $\pa_x u(0)= -\pa_y v(0)$ due to divergence free $\na \cdot \uu =0$. For $\om = O(|\xx|^2)$ with $\om \in C^{2,1}$, we obtain
$ ( u - \pa_x u(0) x, v + \pa_x u (0) y ) = O(|\xx|^3)$, and the first part vanishes $O(|\xx|^3)$ near $\xx=0$. See details at the beginning of \cite[Section 4.2]{ChenHou2023a}.

For the nonlinear term $\cN_1(\om,\eta,\xi) = N(\om) $ in \eqref{eq:linomega}, for 
perturbation with regularity: $\om, \eta, \xi = O(|\xx|^2)$, $\om, \na \th \in C^{2, 1}$, and $\om, \eta$ being odd in $x$, $\xi$ being even in $x$, and $|\xi| \les x^2$
\footnote{
$\xi$ is even in $x$ and $ \xi(0, y) = 0, \pa_x \xi(0, y)$ are preserved by the equations.
}
, we obtain
\[
   \cN_1(\om , \eta, \xi) = c_{\om}(\om) \pa_{xy} \om(0) x y + O(|\xx|^3) .
\]

Therefore, by subtracting an operator in the form of $ \pa_{xy} \olin \cF_1(0) \cdot f_1 
+ c_{\om}( \cdot ) f_2 + c_{\om}(\om) \pa_{xy}\om(0) f_3 $ in \eqref{eq:linomega} 
with $f_i$ having appropriate leading order near $\xx=0$ and performing a 
decomposition similar to \eqref{eq:lin_W} and \eqref{eq:lin_Wm}, we obtain the equation for $\om_1$ that preserves $O(|\xx|^3)$ vanishing order near $\xx= 0$. While $c_{\om}(\om) \pa_{xy}\om(0)$ is nonlinear in $\om$, since it is constant in space, we can still apply the Duhamel formula \eqref{eq:W2_form}, \eqref{eq:W2_form_hat} to the source term $c_{\om}(\om) \pa_{xy} \om(0) f_3$. 
We refer to \cite[Section 4.2]{ChenHou2023a} for the detailed decomposition for \eqref{eq:lin}.

\subsection{Remark on local well-posedness}

The decomposition and vanishing order for the perturbation of \eqref{eq:lin} are preserved as long as the equation is locally well-posed with a sufficiently regular solution, e.g.
 $\om, \th \in C^{4,1}$ and $\om_1, \eta_1, \xi_1 \in C^{2,1}$, which is a \emph{qualitative} property. 
 Since we consider sufficiently smooth initial data, the local-in-time regularity of these variables 
 follows from standard local regularity for the Boussinesq equations. 
 Since this paper is mainly focused on the analytic, low-rank correction, we refer to  \cite[Section 4.4]{ChenHou2023a} for more discussion on the regularity. 

 Using Beale-Kato-Majda type continuation criterion, we \emph{only} need to control $\int_0^T \| \na (\uu, \th)(s) \|_{L^{\infty}} ds $ to propagate these  \emph{qualitative} regularity properties. 
By performing stability estimates on $W_1 = (\om_1, \eta_1, \xi_1), c_{\om}(\om)$
and $\pa_{xy}\om(0),  \pa_{xxy}\th(0)$ (via ODEs), 
we obtain smallness estimates of $W_1$ uniform in time, which enables us to control $\wh W_2$ using the solution formula \eqref{eq:W2_form_hat}. These stability estimates further control the continuation criterion. 
Then by a bootstrap argument, we obtain global stability estimates.
Denote $W_1 = (\om_1, \eta_1, \xi_1),\wh W_2= (\hat \om_2, \hat \eta_2, \hat \xi_2)$. In summary, we have the following logical cycle:
 \[
 \bal
    \mbox{(a) \emph{qualitative} local regularity of $\om, \th$} & \Rightarrow 
    \mbox{(b) \emph{qualitative} local regularity of $W_1, \wh W_2$}, \\
   \mbox{(c) \emph{quantitative} weighted stability estimates for $W_1$}
 & \Rightarrow   \mbox{(d) \emph{quantitative} stability estimates for $W_1, \wh W_2$ } \\
  &  \Rightarrow \mbox{(e) continuation criterion on $\om, \th$}     \Rightarrow (a).
\eal
 \]
The key step is to establish (c) \emph{quantitative} weighted stability estimates for $W_1$, which 
are the main results in \cite{ChenHou2023a,ChenHou2023b}.

\section{Numerical velocity construction and analytic correction of the error}
\label{sec:numer_vel}

The model problem \eqref{eq:lin_model} explains the local analytic correction near $0$ for a 
numerically constructed solution.
For the full Boussinesq equation \eqref{eq:lin}, there is another important analytic correction. This is the elliptic counterpart of the space-time correction discussed in Section~\ref{sec:model}. The approximation to the nonlocal velocity field of the profile is constructed numerically through the stream function. The approximation error for the \emph{nonlocal} velocity is 
estimated from the \emph{local} error in solving the Poisson equation via weighted functional inequalities. We employ ideas and methods similar to those in Section \ref{sec:model} to perform low-rank analytic correction for the error to ensure that it belongs to suitable weighted energy spaces.

\paragraph{\bf Step 1. Numerical construction}

In \cite[Section 7]{ChenHou2023a}, the approximate vorticity profile $\bar\omega$ is first
constructed numerically in an explicit basis representation like 
\eqref{eq:ass_f_bar:a}. 
We construct a numerical stream function $\bar \phi^{N,(0)}$ by solving
\[
-\Delta \bar\phi = \bar\omega,
\]
numerically and representing $\bar\phi^{N, (0)}$ using basis functions like \eqref{eq:ass_f_bar:a},
where the superscript \textit{N} indicates numerics. The construction of $\bar \phi^{N, (0)}$ introduces a \emph{local} error 
\[
\bar\varepsilon^{(0)}=\bar\omega-(-\Delta)\bar\phi^{N, (0)}.
\]

\paragraph{\bf Step 2: analytic correction}
Due to numerical error, the error $\bar\varepsilon^{(0)} \neq 0$. Due to the odd symmetry of 
$\bar \om, \bar \phi^{N, (0)}$ in $x$, we obtain 
\[
\bar\varepsilon^{(0)} = \pa_{x} \bar\varepsilon^{(0)}(0) x + O(|\xx|^2), \qquad \text{near } \xx=0.
\]
 The derivative $\pa_x \bar\varepsilon^{(0)}(0)$ at $\xx=0$ is defined 
\emph{analytically} via the basis representation for $ \bar \phi^{N, (0)}$ and $\bar \om$ like 
\eqref{eq:F_cor1_analy}. Similar to \eqref{eq:g_cor1}, we perform an analytic correction to improve vanishing order. We define 
\[
  \chi_3 = - \f{x y^2}{2} \kp_*(x) \kp_*(y), 
\]
for some sufficiently smooth function $\kp_*(x)$ with $\kp_*(y)=  1 +  O(|y|^4)$ and compact support near $y=0$. It satisfies 
\[
 -\D \chi_3 = x + O(|\xx|^4) , \qquad \text{near } \xx=0 .
\] 

Then we perform an analytic correction to $\bar \phi^{N, (0)}$, exactly in the same spirit as the first correction for $\wh g^{(0)}$ in Section~\ref{sec:numer}:
\[
\bar \phi^N = \bar \phi^{N, (0)} + \pa_x \bar\varepsilon^{(0)}(0)  \cdot \chi_3.
\]

The error associated with the corrected numerical stream function $\bar \phi^N$ satisfies
\bseq\label{eq:nonloc_err_cor1}
\beq
\bar \e \teq \bar \om - (-\D \bar \phi^N)   = \bar \e^{(0)} +  \pa_x \bar\varepsilon^{(0)}(0)  
\D \chi_3 
 = \bar \e^{(0)} -  \pa_x \bar\varepsilon^{(0)}(0) x + O(|\xx|^4)
 = O(|\xx|^2),  \   \text{near } \xx=0.
\eeq

Due to the odd symmetry of $\bar \om, \bar \phi^N$ in $x$, we obtain $\pa_x^2 \bar \e(0)
= \pa_y^2 \bar \e(0) = 0$ and 
\beq\label{eq:nonlocal_err_cor1}
 \bar \e = \pa_{xy} \bar \e(0) x y + O(|\xx|^3) , \quad  \text{near } \xx=0.
\eeq
\eseq

We define the associated numerical velocity analytically by applying $\na^{\perp}$ to the basis functions:
\beq\label{eq:u_bar_N}
\bar \uu^N = \nabla^\perp \bar\phi^N.
\eeq

\paragraph{\bf Weighted functional inequalities}

Due to the error $\bar \e$, the velocity $\bar \uu^N$ constructed numerically differs from the 
exact velocity $\uu(\bar \om) = \na^{\perp}(-\D)^{-1} \bar \om$. To control the nonlocal error  rigorously, we first rewrite the error analytically
\[
 \uu(\bar \om ) - \bar \uu^N = \uu(\bar \om - (-\D) \bar \phi^N )
 = \uu( \bar \e).
\]

In \cite{ChenHou2023a}, \cite[Section 4]{ChenHou2023b}, 
we establish weighted functional inequalities.  
One of these inequalities takes the following form 
\footnote{
In \cite[Section 4.3]{ChenHou2023a}, we approximate $u(f)$ by a finite rank operator 
$\hat u(f)$, which is bounded for $\| f \vp \|_{L^{\infty}} < \infty$. To explain the idea,
we simplify the operator $\hat u(f)$ by $\pa_x u(f)(0) x$ in \eqref{eq:non_ineq}.
}
\beq\label{eq:non_ineq}
  | ( u(f)- \pa_x u(f)(0) \cdot x ) \rho_{10} | \leq C(|\xx|) \| f \vp \|_{L^{\infty}} ,
\eeq
where $u$ is the first component of vector $\uu$, for some bounded coefficients $C(|\xx|)$, with weight 
\[
 \rho_{10} \sim |\xx|^{-3}, \quad \vp \sim |\xx|^{-1/2} |\xx|^{-2}, \quad \text{near } \xx = 0.
\]

\paragraph{\bf Step 3: analytic decomposition of the error}

To apply the weighted estimates \eqref{eq:non_ineq}, we require the input $f$ to be odd in $x$ and to vanish like $O(|\xx|^{5/2})$ or better near $\xx=0$. Yet, the error $\bar \e$ fails to have the vanishing order required by the weighted norm. Therefore, we perform another finite-rank analytic decomposition of $\bar \e$. The purpose of this decomposition is to separate the unique mixed quadratic defect mode from the remainder that will satisfy the improved vanishing order. 

In \cite[Section 5.8]{ChenHou2023a}, \cite[Section 3]{ChenHou2023b}, we treat both the steady-state
stream-function error $\bar\varepsilon$ \eqref{eq:nonloc_err_cor1} and the analogous space-time error
$\hat\varepsilon$ in solving \eqref{eq:lin_g} in the same way. For an error $\varepsilon=\bar\varepsilon$ or $\hat\varepsilon$, one chooses a cutoff
function $\chi_\varepsilon$ satisfying
\beq\label{def:chi_e}
\chi_\varepsilon = 1 + O(|\xx|^4)
\qquad \text{near } \xx=0,
\eeq
and decomposes
\[
\varepsilon = \varepsilon_1 + \varepsilon_2.
\]
The low-rank correction term is chosen explicitly so that $\varepsilon_2$ captures the unique mixed quadratic defect mode, and we can derive the  velocity field explicitly and analytically, while $\varepsilon_1$ is the remainder with improved vanishing order. Concretely, one chooses
\footnote{
We note a typographical error in the basis function for $\e_2$ in an early version of \cite[Section 5.8]{ChenHou2023a}, where it is written as $\e_2=\pa_{xy}\e(0)\,\D\! (\f{x^3 y}{2}\chi_{\e} )$. The correct definition is given in \eqref{def:e2} and is the one used in the computer-assisted proof. In particular, this error has no effect on the validity of the proof.
}
\beq\label{def:e2}
\varepsilon_2
= \pa_{xy} \varepsilon(0) \cdot (-\Delta)\! \big( - \frac16 x y^3\,\chi_\varepsilon \big),
\eeq
where $\chi_{\e}$ is defined in \eqref{def:chi_e}. Due to the leading order behavior of $\chi_{\e}$ \eqref{def:chi_e}, we have 
\[
  \e_2 = \pa_{xy} \e(0) \cdot \D( \f{1}{6} x y^3 ) + O(|\xx|^4)
 =  \pa_{xy} \e(0)  \cdot x y +  O(|\xx|^4) ,  \qquad \text{near } \xx=0.
\]

Let us discuss the vanishing order of $\e_1$. 
Due to the odd symmetry in $x$, the total error $\e$ 
satisfies a Taylor expansion:
\[
 \e = \pa_{xy} \e(0) x y + O(|\xx|^3) ,  \qquad \text{near } \xx=0 .
\]

Thus, combining the above two estimates, we derive 
\[
 \e_1 = \e - \e_2 =   \e - \pa_{xy} \e(0) \cdot \D \big( \frac16 x y^3\,\chi_\varepsilon \big)
 = O(|\xx|^3), 
\qquad \text{near } \xx=0 .
\]

Using the structure \eqref{def:e2}, we derive the velocity $\uu(\e_2)$ \emph{analytically} 
\beq\label{def:ue2}
\uu(\e_2) = \na^{\perp}(-\D)^{-1} \e_2 = \pa_{xy} \e(0) \cdot \na^{\perp} \big( - \tfrac16 x y^3\,\chi_\varepsilon \big) = O(|\xx|^3) ,
 \quad \text{near } \xx=0 , \  \na \uu(\e_2)(0) = 0 .
\eeq
Due to the cubic vanishing, $\uu(\e_2)$ belongs to the weighted energy space for the velocity in our energy estimates. We decompose the velocity error $\uu(\e) = \na^{\perp}(-\D)^{-1} \e$ as
\[
\uu(\e) = \uu(\e_1) + \uu(\e_2)
= \uu(\e_1) + \pa_{xy} \e(0) \cdot \na^{\perp} \big( - \tfrac16 x y^3\,\chi_\varepsilon \big).
\]

The corrected residual $\varepsilon_1$ now has the cubic vanishing order required for the weighted
estimates in \eqref{eq:non_ineq}. We obtain piecewise weighted bounds for $\e_1 \vp$ and apply functional inequality \eqref{eq:non_ineq} to bound $u(\e_1) - \pa_x u(\e_1)(0) x$. For $ \uu(\e_2)$, we bound it directly using its local structure \eqref{def:ue2} and piecewise bounds for $ \chi_\varepsilon$. 
Since $\na \uu(\e_2)(0) = 0$ \eqref{def:ue2}, we obtain $ \pa_x u(\e_1)(0) = \pa_x u(\e)(0)$. We estimate $\pa_x u(\e)(0) = \pa_{xy} (-\D)^{-1}( \e)(0)$ by bounding the integral directly.

\medskip
\noindent
\textbf{Vanishing order for nonlocal error.}
The key points for the rigorous low-rank correction near the origin for the 
numerical stream-function are as follows.

\begin{itemize}[leftmargin=2.5em]
\item The residual $\varepsilon$ is an analytically defined function with sufficient regularity once the
      numerical stream function $\bar\phi^{N, 0}$ is fixed and represented using basis functions like \eqref{eq:ass_f_bar:a}.
\item The coefficient $\varepsilon_{xy}(0)$ is the exact analytic mixed
      derivative of that residual at the origin.
\item The correction term $\varepsilon_2$ is built from an explicit analytic
      expression whose velocity field is derived \emph{analytically}.

\item  The remainder $\varepsilon_1$ is an analytically defined function
with improved vanishing order, and therefore belongs to the weighted
space required for inequalities such as \eqref{eq:non_ineq}.

\end{itemize}

This is why the additional approximation in the velocity construction does not 
destroy the vanishing conditions near the origin.  It is handled by the same
analytic low-rank correction principle as in the toy model.

\section{Conclusion}

In this review, we revisited the analytic finite-rank correction method in the computer-assisted stability analysis developed in \cite{ChenHou2023a,ChenHou2023b}. Our main goal was to isolate one of the key structural features of the analysis: the interplay between numerically constructed functions, represented in explicit basis functions, and exact local analytic constraints imposed by singularly weighted estimates.

We applied this analytic low-rank correction both to the construction of approximate space-time solutions in the quantitative finite-rank perturbation method (see Section \ref{sec:finite_rank}) and to the construction of the stream function together with the estimation of its residual error. In each case, numerical approximations may achieve very small global error with small derivatives, but they do not automatically satisfy the exact local vanishing conditions required by the singularly weighted estimates. The analytic correction serves precisely to remove this mismatch.

The key point is that the numerical step \emph{only} determines coefficients in explicit basis representations. Once these coefficients are fixed, they determine globally defined functions on the entire domain. The subsequent low-rank corrections are then carried out \emph{analytically} on these functions. 
We applied the same analytic principle throughout the numerical construction of the approximate profile, the approximate space-time solutions, and the approximate stream function:
\begin{itemize}[leftmargin=2.5em]
\item numerical coefficients determine explicit globally defined functions with sufficient regularity;
\item one identifies the remaining bad low-order mode near the origin via Taylor expansions; 
\item one introduces an explicit finite-rank correction analytically;
\item the corrected function satisfies the desired vanishing condition exactly;
\item one estimates the corrected function or its error in the singularly weighted norm.
\end{itemize}
In particular, the improved vanishing order becomes an \emph{exact} analytic property of the corrected function, even though the coefficients of the approximate functions are obtained numerically. This is the basic principle behind the decomposition of the perturbation near the origin in \cite{ChenHou2023a,ChenHou2023b}.

This review highlights this aspect of the proof strategy and makes the underlying idea transparent for future applications. Similar combinations of explicit global approximations and exact local corrections may also prove useful in other computer-assisted proofs of singularity formation and nonlinear stability problems for fluid equations and related PDEs.

\vspace{0.2in}
\noindent
{\bf Acknowledgments.} 
The research of J. Chen was in part supported by the NSF Grant DMS--2408098. The research of T. Y. Hou was in part supported by the NSF grants DMS-2205590, DMS-2508463, the Choi Family Gift Fund and the Dr. Mike Yan Gift Fund.

\bibliographystyle{plain}
\bibliography{selfsimilar}

\end{document}